\documentclass[12pt]{amsart}

\usepackage{amssymb, amsfonts}
\usepackage{url}
\usepackage[all,arc]{xy}
\usepackage{amscd}
\usepackage{mathrsfs}
\usepackage{geometry}
\usepackage[colorlinks,citecolor=blue]{hyperref}
\usepackage{comment}
\usepackage{enumerate}
\usepackage{graphicx}
\usepackage{bm}
\usepackage{color}
\usepackage{anyfontsize}
\usepackage{caption}

\numberwithin{equation}{section}

\theoremstyle{plain}
\newtheorem{theorem}{Theorem}[section]

\newtheorem{lemma}[theorem]{Lemma}
\newtheorem{proposition}[theorem]{Proposition}

\newtheorem{definition}[theorem]{Definition}

\theoremstyle{remark}
\newtheorem{remark}{Remark}[section]

\begin{document}


\title[On FBMS with boundary on concentric spheres]{On free boundary minimal submanifolds with boundary on concentric spheres in Euclidean space}

\author{Tianyu Ma, Vladimir Medvedev}

\address{Faculty of Mathematics, National Research University Higher School of Economics, 6 Usacheva Street, Moscow, 119048, Russian Federation}

\email{tma@hse.ru, vomedvedev@hse.ru}



\begin{abstract}
The search for free boundary minimal submanifolds in Euclidean space with boundaries on a collection of concentric spheres naturally extends the classical problem in the unit Euclidean ball. A key feature of this setting is that the coordinate functions of such submanifolds satisfy a \textit{Steklov problem with an indefinite weight}. This framework allows us to introduce a spectral index, which in turn yields both upper and lower bounds for the Morse index. As a concrete application, we compute the exact Morse index of an $m$-dimensional flat annulus in an $n$-dimensional spherical shell, showing that it equals $n-m$. Moreover, 
we study in detail free boundary minimal immersions from 2-dimensional annuli into Euclidean space whose boundaries lie on concentric spheres. We show that the images of these free boundary minimal immersions (FBMI) lie in an $m$-dimensional subspace with $2\leqslant m\leqslant 4$, and list the explicit forms of these FBMI. We also demonstrate how to find examples of FBMIs for which the ratio between the radii of the concentric spheres containing the boundaries is arbitrarily large.
 \end{abstract}

\maketitle


\newcommand\cont{\operatorname{cont}}
\newcommand\diff{\operatorname{diff}}

\newcommand{\dvol}{\text{dA}}
\newcommand{\Ric}{\operatorname{Ric}}
\newcommand{\Hess}{\operatorname{Hess}}
\newcommand{\GL}{\operatorname{GL}}
\newcommand{\myO}{\operatorname{O}}
\newcommand{\myP}{\operatorname{P}}
\newcommand{\eye}{\operatorname{Id}}
\newcommand{\myF}{\operatorname{F}}
\newcommand{\Vol}{\operatorname{Vol}}
\newcommand{\odd}{\operatorname{odd}}
\newcommand{\even}{\operatorname{even}}
\newcommand{\ol}{\overline}
\newcommand{\mye}{\operatorname{E}}
\newcommand{\myo}{\operatorname{o}}
\newcommand{\myt}{\operatorname{t}}
\newcommand{\irr}{\operatorname{Irr}}
\newcommand{\mydiv}{\operatorname{div}}
\newcommand{\curl}{\operatorname{curl}}
\newcommand{\re}{\operatorname{Re}}
\newcommand{\im}{\operatorname{Im}}
\newcommand{\can}{\operatorname{can}}
\newcommand{\scal}{\operatorname{scal}}
\newcommand{\tr}{\operatorname{trace}}
\newcommand{\sgn}{\operatorname{sgn}}
\newcommand{\SL}{\operatorname{SL}}
\newcommand{\myspan}{\operatorname{span}}
\newcommand{\mydet}{\operatorname{det}}
\newcommand{\SO}{\operatorname{SO}}
\newcommand{\SU}{\operatorname{SU}}
\newcommand{\specl}{\operatorname{spec_{\mathcal{L}}}}
\newcommand{\fix}{\operatorname{Fix}}
\newcommand{\id}{\operatorname{id}}
\newcommand{\grad}{\operatorname{grad}}
\newcommand{\singsup}{\operatorname{singsupp}}
\newcommand{\wave}{\operatorname{wave}}
\newcommand{\ind}{\operatorname{ind}}
\newcommand{\mynull}{\operatorname{null}}
\newcommand{\inj}{\operatorname{inj}}
\newcommand{\arcsinh}{\operatorname{arcsinh}}
\newcommand{\Spec}{\operatorname{Spec}}
\newcommand{\Ind}{\operatorname{Ind}}
\newcommand{\Nul}{\operatorname{Nul}}
\newcommand{\inrad}{\operatorname{inrad}}
\newcommand{\mult}{\operatorname{mult}}
\newcommand{\Length}{\operatorname{Length}}
\newcommand{\Area}{\operatorname{Area}}
\newcommand{\Ker}{\operatorname{Ker}}
\newcommand{\floor}[1]{\left \lfloor #1  \right \rfloor}

\newcommand\restr[2]{{
  \left.\kern-\nulldelimiterspace 
  #1 
  \vphantom{\big|} 
  \right|_{#2} 
  }}


\section{Introduction}

Since its introduction in the 1940s through the seminal works of Courant~\cite{courant1940existence,courant1945plateau}, the study of free boundary minimal submanifolds in a given Riemannian manifold has become a classical problem in geometric analysis. Recall that a submanifold with boundary $\Sigma$ in a Riemannian manifold $(M,h)$ with boundary is called \emph{free boundary minimal} if its mean curvature vector vanishes identically and $\Sigma$ meets $\partial M$ orthogonally along $\partial\Sigma$. Equivalently, $\Sigma$ is a critical point of the volume functional with respect to variations whose variation vector fields are tangent to $\partial M$ along $\partial\Sigma$. The specific case of free boundary minimal submanifolds in the unit ball $\mathbb B^n$ of Euclidean space $\mathbb E^n$ has garnered significant attention due to its profound connection to the spectral geometry of the Steklov eigenvalue problem (see, e.g.,~\cite{li2019free}, \cite[Chapter 1]{fraser2020geometric}, and the references therein). More recently, further results building upon this foundational connection have been established in~\cite{lima2023eigenvalue,medvedev2025free,ma2025some}, extending the study of free boundary minimal submanifolds to geodesic balls in spherical and hyperbolic spaces.

In this paper, we investigate a natural generalization of the free boundary minimal submanifold problem in the unit ball: we consider free boundary minimal submanifolds whose boundaries are constrained to lie on a collection of concentric spheres. This geometric setting has been previously explored, for instance, in~\cite{ambrozio2025intrinsic}.

We now state the precise problem. Let $\Phi \colon \Sigma \to \mathbb{E}^n$ be a free boundary minimal immersion (FBMI) whose boundary lies on a collection of concentric spheres with radii $R_\alpha$ ($\alpha=1,\ldots,A$) and $r_\beta$ ($\beta=1,\ldots,B$). We use capital letters $R_\alpha$ to denote the radii of the spheres where $\Sigma$ meets the boundary from the convex side, and lowercase letters $r_\beta$ for those where it meets the boundary from the concave side. By a slight abuse of notation, we identify $\Sigma$ with its image $\Phi(\Sigma)$.

Let $\partial\Sigma = \bigcup_\alpha \partial_+^\alpha\Sigma \cup \bigcup_\beta \partial_-^\beta\Sigma$, where $\bigcup_\alpha \partial_+^\alpha\Sigma$ denotes the union of connected components of $\partial\Sigma$ intersecting the spheres $\mathbb{S}^{n-1}(R_\alpha)$, and $\bigcup_\beta \partial_-^\beta\Sigma$ denotes the union of connected components intersecting the spheres $\mathbb{S}^{n-1}(r_\beta)$. It is straightforward to verify that the coordinate functions $\phi_i$ ($i=1,\ldots,n$) of the immersion satisfy the following boundary value problem:
\begin{align}
\label{eqn:general}
\begin{cases}
\Delta_g \phi_i = 0 & \text{in } \Sigma, \\[6pt]
\dfrac{\partial \phi_i}{\partial\eta} = -\dfrac{1}{r_\beta}\phi_i & \text{on } \partial_-^\beta\Sigma, \\[6pt]
\dfrac{\partial \phi_i}{\partial\eta} = \dfrac{1}{R_\alpha}\phi_i & \text{on } \partial_+^\alpha\Sigma,
\end{cases}
\end{align}
where $g$ is the metric induced on $\Sigma$ by the Euclidean metric, $\Delta_g = -\text{div}_g \nabla^g$ is the corresponding Laplace--Beltrami operator, and $\eta$ denotes the unit outward conormal vector field along $\partial\Sigma$.

Problem~\eqref{eqn:general} is a specific instance of the \textit{Steklov problem with an indefinite weight} (\cite{sandgren1955vibration,suslina1999asymptotics,suslina1999spectral,agranovich2006mixed}):
\begin{align}\label{eq:steklov_weight}
\begin{cases}
\Delta_g u = 0 & \text{in } \Sigma, \\
\dfrac{\partial u}{\partial \eta} = \rho \sigma u & \text{on } \partial\Sigma,
\end{cases}
\end{align}
where the weight function $\rho$, which is not identically zero, satisfies the integrability condition
\[
\int_{\partial\Sigma} |\rho| \log(2+|\rho|) \,ds_g < \infty.
\]
Under this assumption on $\rho$ (which we typically assume changes sign), the spectrum of this problem is discrete and consists of a sequence of eigenvalues (see \cite[Proposition 2.2]{karpukhin2023weyl} and \cite[Sections 2 and 3]{girouard2021continuity}):
\[
-\infty \swarrow \dots \leqslant \sigma^-_i \leqslant \dots \leqslant \sigma^-_1 < 0 = \sigma_0 < \sigma^+_1 \leqslant \dots \leqslant \sigma^+_i \leqslant \dots \nearrow +\infty,
\]
where we abbreviate $\sigma^\pm_i(\Sigma,g,\rho)$ to $\sigma^\pm_i$ when the manifold $(\Sigma,g)$ and weight $\rho$ are clear from the context. Furthermore, each eigenvalue has finite multiplicity.

In our setting, this weight function is explicitly given by
\begin{align}\label{rho}
\rho = \begin{cases}
\dfrac{1}{R_\alpha} & \text{on } \partial_+^\alpha\Sigma, \\[6pt]
-\dfrac{1}{r_\beta} & \text{on } \partial_-^\beta\Sigma.
\end{cases}
\end{align}

Motivated by analogous problems in the unit Euclidean ball, we seek examples such that $g$ is an $\mathbb{S}^1$-symmetric metric on the topological annulus. In \cite{fraser2011first}, the authors constructed $\mathbb{S}^1$-symmetric free boundary minimal annuli and Möbius bands, notably including the embedded critical catenoid and critical Möbius band. Our aim is to construct analogous examples with boundaries lying on concentric spheres. However, this restricts the number of boundary spheres: an annulus can have at most two boundary spheres, while a Möbius band can have only one. The case of a single boundary sphere corresponds to an FBMI inside a single ball, which has been extensively studied. Therefore, we focus on free boundary minimal annuli whose two boundary components lie on two distinct concentric spheres.

Since the coordinate functions $\phi_i$ of such an annulus $\Phi \colon \Sigma \to \mathbb{E}^n$ are harmonic, the squared norm $|\Phi|^2$ is subharmonic on $(\Sigma,g)$. By the maximum principle, $|\Phi|^2$ attains its maximum on the boundary $\partial\Sigma$. Let $\mathbb{S}^{n-1}(r)$ and $\mathbb{S}^{n-1}(R)$ (with $r<R$) be the two concentric spheres containing $\partial\Sigma$. The maximum principle implies that $\Sigma \subset \mathbb{B}^n(R)$, meaning $\Sigma$ must intersect the larger sphere $\mathbb{S}^{n-1}(R)$ from the inside (i.e., from the convex side). The smaller sphere, however, can be intersected by $\Sigma$ either from the inside or the outside, a geometric distinction encoded by a sign $\epsilon \in \{-1, 1\}$. (Note that this is a general phenomenon: any FBMI with boundary on a collection of concentric spheres must intersect the outermost sphere from the inside, i.e., from the convex side.)

Consequently, the coordinate functions satisfy the following boundary value problem:
\begin{align}
\label{eq:annulus}
\begin{cases}
\Delta_g \phi_i = 0 & \text{in } \Sigma, \\[6pt]
\dfrac{\partial \phi_i}{\partial\eta} = \dfrac{\epsilon}{r}\phi_i & \text{on } L_1, \\[6pt]
\dfrac{\partial \phi_i}{\partial\eta} = \dfrac{1}{R}\phi_i & \text{on } L_2,
\end{cases}
\end{align}
where $L_1$ and $L_2$ denote the boundary components of $\Sigma$ lying on $\mathbb{S}^{n-1}(r)$ and $\mathbb{S}^{n-1}(R)$, respectively.

Our first result, which builds upon the techniques introduced in \cite{fraser2011first,fan2015extremal}, is the following:
\begin{theorem}
\label{thm:fbmi cases}
Let $\Sigma=\mathbb{S}^1\times [0,T]$ be a standard flat annulus with its canonical metric $g_{0,T}$. For any $0<r\leqslant R$ and any rotationally symmetric metric $g$ conformal to $g_{0,T}$, suppose that there exists a minimal immersion $\Phi:(\Sigma^2,g) \rightarrow \mathbb{E}^{n}$ with free boundaries on the two concentric spheres $\mathbb{S}^{n-1}(r)$ and $\mathbb{S}^{n-1}(R)$. Then at least one component of $\partial\Sigma$ intersects one of the spheres (the larger sphere if $r<R$) from the inner side. Through an orthogonal transformation of the ambient space $\mathbb{E}^{n}$, we may assume that $2\leqslant n\leqslant 4$. Furthermore, we have the following.
\begin{itemize}
\item For $\epsilon=-1$, one component of $\partial\Sigma$ intersects $\mathbb{S}^{n-1}(r)$ from the outer side, and the other component intersects $\mathbb{S}^{n-1}(R)$ from the inner side with $r<R$. In this case, $\Phi$ is two-dimensional and is given by 
\begin{align}
\label{eqn:FBMI annulus 2-dim}
 \Phi(t,\theta)=\begin{pmatrix}
     f(t)\cos(k\theta)\\
     f(t) \sin(k\theta)
 \end{pmatrix},   
\end{align}
where $f(t)=a_k\cosh(kt)+b_k\sinh(kt)$ with $a_k,b_k\in\mathbb{R}$, $k\geqslant 1$ and $a_kb_k>0$.
\item For the case $\epsilon=1$, $\partial\Sigma$ intersects both spheres from the inner side. The image is three-dimensional or four-dimensional.
\begin{itemize}
    \item If zero-frequency eigenfunctions (linear functions in $t$) are in the span of the coordinate functions of $\Phi$, the image is three-dimensional. Through a similarity transformation of the ambient space $\mathbb{E}^{n}$, $\Phi$ has the form
    \begin{align}
    \label{eqn:FBMI annulus 3-dim}
      \Phi(t,\theta)=\begin{pmatrix}
          f(t) \cos(k\theta)\\
          f(t) \sin (k\theta)\\
          t+c_2
      \end{pmatrix}.  
    \end{align}
   \item If only eigenfunctions of positive frequencies are allowed to span the coordinate functions of $\Phi$, the image is four-dimensional. Through an orthogonal transformation of the ambient space $\mathbb{E}^{n}$, $\Phi$ has the form
   \begin{align}
   \label{eqn:FBMI annulus 4-dim case}
    \Phi(t,\theta)=\begin{pmatrix}
        f_1(t)\cos(k_1\theta)\\
        f_1(t)\sin(k_1\theta)\\
        f_2(t)\cos(k_2\theta)\\
        f_2(t)\sin (k_2\theta)
    \end{pmatrix},
    \end{align}
    where $f_i(t)=a_i\cosh(k_i t)+b_i\sinh(k_i t)$, for some $k_1,k_2\geqslant 1$ and $a_i,b_i\in\mathbb{R}$.
    \end{itemize}
\end{itemize}
\end{theorem}
The coordinate functions of each FBMI $\Phi$ in the theorem above are separable Steklov eigenfunctions. The parameters in the Steklov eigenfunctions $f$ or $f_i$ determine the possible forms of the FBMIs and affect asymptotics of the radii ratio $\dfrac{R}{r}$ in the following way.
\begin{theorem}
\label{thm:fbmi cases-2}
Let $\Phi: \mathbb{S}^1\times [0,T]\rightarrow \mathbb{E}^n$ be one of the canonical forms  of the FBMI (\ref{eqn:FBMI annulus 2-dim}--\ref{eqn:FBMI annulus 4-dim case}) in theorem \ref{thm:fbmi cases}. We have the following:
\begin{itemize}
    \item The two-dimensional FBMI \eqref{eqn:FBMI annulus 2-dim} exists for any conformal class $T>0$ given any parameters $k\geqslant 1$ and $a_kb_k>0$. Moreover, for any $k\geqslant 1$ and $T>0$, we have $\dfrac{R}{r}\rightarrow\infty$ as $\dfrac{b_k}{a_k}\rightarrow\infty$.
    \item For any $k\geqslant 1$, the three-dimensional FBMI \eqref{eqn:FBMI annulus 3-dim} exists for $T\in \left[\dfrac{T_0}{k},\infty\right)$, where $T_0$ is the constant given by \eqref{eqn: T0 value}. For any fixed $k\geqslant 1$, $\Phi$ in the form of \eqref{eqn:FBMI annulus 3-dim} can be chosen so that $\dfrac{R}{r}\rightarrow\infty$ as $T\rightarrow\infty$.
    \item For any fixed $k_1\geqslant 1$ and $T>0$, for all sufficiently large $k_2>k_1$, the four-dimensional FBMI of the form \eqref{eqn:FBMI annulus 4-dim case} exists. Moreover, for any $k_1\geqslant 1$ and $T>0$ fixed, as $k_2\rightarrow \infty$, the map $\Phi$ can be chosen so that $\dfrac{R}{r}\rightarrow \infty$.
\end{itemize}
\end{theorem}
Our second set of results are related to the (Morse) index of an FBMI with boundaries in a collection of concentric spheres. The standard formulas for the second variation of the volume and energy functionals yield the following expressions:

1) The second variation of the volume functional with respect to a normal variation field $V$ is given by
\begin{equation}\label{index_S}
\begin{split}
S(V,V) &= \int_\Sigma \left( |\nabla^\perp V|^2 - \mathcal {B}(V) \right) \,dv_g \\
&\quad - \sum_\alpha \frac{1}{R_\alpha} \int_{\partial_\alpha\Sigma} |V|^2 \,ds_g + \sum_\beta \frac{1}{r_\beta} \int_{\partial_\beta\Sigma} |V|^2 \,ds_g.
\end{split}
\end{equation}
where $\nabla^\perp$ denotes the normal connection on the normal bundle of $\Sigma$, and $\mathcal{B}(V)$ represents the Simons operator. Recall that $\mathcal{B} \equiv 0$ whenever $\Sigma$ is totally geodesic.

2) The second variation of the energy functional with respect to a variation vector field $X$ is given by
\[
S_E(X,X)=\int_\Sigma|\nabla X|^2\,dv_g-\sum_\alpha \frac{1}{R_\alpha}\int_{\partial_\alpha\Sigma}|X|^2\,ds_g+\sum_\beta \frac{1}{r_\beta}\int_{\partial_\beta\Sigma}|X|^2\,ds_g,
\]
where $\nabla$ denotes the Levi-Civita connection associated with the Euclidean metric.

We denote the index of the energy and volume functionals of $(\Sigma,g)$ by $\Ind_E(\Sigma)$ and $\Ind(\Sigma)$, respectively. 

The primary novelty of this part of the article lies in adapting the framework developed in~\cite{medvedev2023index,medvedev2025free} to geometric setting of FBMI with boundaries in a collection of concentric spheres through the introduction of the spectral index. As detailed in Definition~\ref{def:ind_spec}, the spectral index $\Ind_S(\Sigma)$ of an FBMI $\Sigma$ is defined as the number of non-negative eigenvalues that are strictly less than one.

We establish the following main results.

\begin{theorem}\label{thm:ind_M}
Let $\Sigma$ be a $2$-dimensional compact free boundary minimal surface with boundary on a collection of concentric spheres in $\mathbb{E}^n$. Then
\[
\Ind(\Sigma) \leqslant n \Ind_S(\Sigma) + \dim \mathcal{M}(\Sigma),
\]
where $\mathcal{M}(\Sigma)$ denotes the moduli space of conformal classes on $\Sigma$.
\end{theorem}

\begin{theorem}\label{+n}
Let $\Sigma$ be a compact non-totally geodesic free boundary minimal hypersurface with boundary on a collection of concentric spheres in $\mathbb{E}^n$. Then
\[
\Ind(\Sigma) \geqslant \Ind_S(\Sigma) + n.
\]
\end{theorem}

Finally, we establish the following result, which completes our analysis.

\begin{theorem}\label{thm:annulus}
 The Morse index of an $m$-dimensional flat annulus, considered as a free boundary minimal submanifold in an $n$-dimensional spherical shell, is exactly $n-m$.
\end{theorem}

We define both a flat annulus and a spherical shell as the region bounded by two concentric spheres. However, we primarily use the term ``spherical shell'' when referring to the ambient domain, while ``flat annulus'' typically refers to the minimal submanifold itself.

In our forthcoming paper, we plan to compute the Morse index of the examples constructed in Theorem~\ref{thm:fbmi cases}.

\subsection*{Acknowledgement}	The study was implemented in the framework of the Basic Research Program at HSE University (HSE-BR-2025-84). The authors would like to thank Viktoria Ipatenkova for her valuable assistance in preparing the results presented in Subsections~\ref{flat} and~\ref{appendix}. The authors also thank Mikhail Karpukhin and Iosif Polterovich for their valuable comments on an earlier version of this manuscript.

\section{FBMI on Annuli}

In this section, we construct examples of $\mathbb{S}^1$-symmetric free boundary minimal annuli within a spherical shell, adapting the strategy employed in \cite{fraser2011first,fan2015extremal}. Throughout this section, $\Sigma$ will always be a $2$-dimensional annulus.

\subsection{Preliminary Results} By the uniformization theorem, any metric on a topological cylinder is conformally equivalent to a unique flat metric on the cylinder $[0,T]\times \mathbb S^1$ with coordinates $(t,\theta)$; the parameter $T>0$ is uniquely determined. The induced metric $g$ on $\Sigma$ is then conformal to $g_T = dt^2 + d\theta^2$, i.e., $g = e^{2\omega(t,\theta)} g_T$. We are interested in $\mathbb S^1$-symmetric free boundary minimal annuli, so $\omega(t,\theta)=\omega(t)$. Clearly,
\[
\Delta_g = e^{-2\omega(t)}\Delta_0,\qquad \Delta_0 = \frac{\partial^2}{\partial t^2} + \frac{\partial^2}{\partial\theta^2},
\]
and the outward normal derivative on the boundaries is given by
\[
\frac{\partial}{\partial\eta}= e^{-\omega(t)}\frac{\partial}{\partial t}\quad\text{on } L_1\ (t=0)\text{ and } L_2\ (t=T),\text{ respectively}.
\]
Then problem \eqref{eq:annulus} can be rewritten as
\begin{align}\label{eq:annulus_conf}
\begin{cases}
\Delta_0 \phi_i=0,&\ i=1,\ldots,n;\\[2pt]
-\dfrac{\partial \phi_i}{\partial t}= e^{\omega(0)}\dfrac{\epsilon}{r}\phi_i,&\ t=0;\\[2pt]
\dfrac{\partial \phi_i}{\partial t}= e^{\omega(T)}\dfrac{1}{R}\phi_i,&\ t=T.
\end{cases}
\end{align}

One can verify that $e^{\omega(0)} = |L_1|_g/(2\pi)$ and $e^{\omega(T)} = |L_2|_g/(2\pi)$. 

Introduce the notation
\begin{align}
\label{eqn:H forms}
\sigma H_0 = \frac{\epsilon}{r}\cdot\frac{|L_1|_g}{2\pi},\qquad 
\sigma H_T = \frac{1}{R}\cdot\frac{|L_2|_g}{2\pi} > 0.
\end{align}
Then problem \eqref{eq:annulus_conf} becomes
\begin{align}\label{eq:annulus_H}
\begin{cases}
\Delta_0 \phi_i=0,&\ i=1,\ldots,n;\\[2pt]
-\dfrac{\partial \phi_i}{\partial t}= \sigma H_0\phi_i,&\ t=0;\\[2pt]
\dfrac{\partial \phi_i}{\partial t}= \sigma H_T\phi_i,&\ t=T.
\end{cases}
\end{align}
Note that $\partial t$ is a constant multiple of $\eta$ on each component of $\partial\Sigma$. Thus, with respect to the metric $g_T$ (not $g$), $\phi_i$ is an eigenfunction of a Steklov problem with weights $H_0$ on $L_1$ and $H_T$ on $L_2$, corresponding to the eigenvalue $\sigma$. For a given metric $g_T$, the values of $H_0$ and $H_T$ satisfying \eqref{eqn:H forms} are determined up to constant scalings: if $\phi_i$ is an eigenfunction with respect to the weights $(H_0,H_T)$ for the flat metric $g_T$ corresponding to the eigenvalue $\sigma$, then it is also an eigenfunction with respect to $(c H_0, c H_T)$ corresponding to the eigenvalue $\sigma/c$ for any $c\neq 0$. Hence, if $H_T\neq 0$, we can always choose $H_T$ to be positive.

The functions $\phi_i$ are also harmonic with respect to the flat metric $g_T$. Using the separation of variables, we may write $\phi_i$ as a linear combination of functions of the form $f(t)u(\theta)$ that satisfy the problem~\eqref{eq:annulus_H}. Then $f$ and $u$ take the form
\begin{equation}
 \label{eqn:solution component}
 \begin{split}
&f(t)=a\cosh(kt)+b\sinh(kt),\quad u(\theta)=A\cos(k\theta)+B\sin(k\theta),\quad k\geq 1;\\
&f(t)=c_1 t+c_2,\quad u=1.
\end{split}
\end{equation}
We refer to the number $k$ as the frequency of a separable eigenfunction; the second line of \eqref{eqn:solution component} corresponds to the case $k=0$. Substituting a function of the form $f(t)u(\theta)$, with $f$ and $u$ as in~\eqref{eqn:solution component}, into~\eqref{eq:annulus_H}, we find that for $k\geq 1$ the boundary conditions become
\begin{align}
\label{eqn:boundary condition flat}
\begin{cases}
 -kb = \sigma H_0 a,\\[2pt]
k\bigl(a\sinh(kT)+b\cosh(kT)\bigr) = \sigma H_T \bigl(a\cosh(kT)+b\sinh(kT)\bigr).
\end{cases}
\end{align}
If $f$ is not identically zero, then $ab \neq 0$.

From \cite{fraser2011first}, system \eqref{eqn:boundary condition flat} reduces to the following quadratic equation:
\begin{align}
\label{eqn:quadratic eq steklov}
\sigma^2 - k\left(\frac{1}{H_0}+\frac{1}{H_T}\right)\coth(kT)\,\sigma + \frac{k^2}{H_0H_T}=0,\qquad k\geq 1.
\end{align}
Since $H_T\sigma > 0$ and, as previously noted, we may choose $H_T > 0$, we only seek positive solutions of this equation. Its discriminant is always positive. Thus each $k \geqslant 1$ yields a pair of distinct eigenvalues $\{\sigma_k^0,\sigma_k^1\}$, given by
\begin{align*}
\sigma_k^0 &= \frac{k}{2}\Bigl((H_0^{-1}+H_T^{-1})\coth(kT) - \sqrt{D}\Bigr),\\
\sigma_k^1 &= \frac{k}{2}\Bigl((H_0^{-1}+H_T^{-1})\coth(kT) + \sqrt{D}\Bigr),
\end{align*}
where
\[
D = (H_0^{-1}+H_T^{-1})^2\coth^2(kT) - \frac{4}{H_0H_T}.
\]
For any fixed $k \geqslant 1$, the eigen-subspace corresponding to $\sigma_k^0$ or $\sigma_k^1$ spanned by eigenfunctions of frequency $k$ and is two-dimensional, with basis $\{ f(t)\cos(k\theta),\, f(t)\sin(k\theta) \}$.

It is possible that equations of the form \eqref{eqn:quadratic eq steklov} for different $k > 0$ share a common solution. Following the approach in~\cite{fan2015extremal,fraser2011first}, we examine the monotonicity of $\{\sigma_k^i\}$ for $i = 0,1$ and $k \geqslant 1$. This determines the dimension of the immersion map $\Phi$. Since the proofs involve technical computations already available in the literature, we only sketch the outline.

For the case $H_0H_T>0$, both sequences $\{\sigma_k^0\}$ and $\{\sigma_k^1\}$ are positive. The same argument as in \cite[Lemma 2.1]{fan2015extremal} shows that both sequences are increasing with respect to $k$. If $H_0H_T<0$, then $\{\sigma_k^0\}$ is negative and $\{\sigma_k^1\}$ is positive. We have $H_T>0$ and $H_0<0$, and we are only interested in the positive weighted Steklov eigenvalues, namely $\{\sigma_k^1\}$. To analyze the monotonicity of $\{\sigma_k^1\}$, we distinguish two subcases based on the sign of $\gamma := H_0^{-1}+H_T^{-1}$.

If $\gamma<0$, we obtain
\begin{align*}
\sigma_k^1 = \frac{2k}{(H_0H_T\gamma)^{-1}\Bigl(\coth(kT)+\sqrt{\coth^2(kT)-\dfrac{4\gamma^2}{H_0H_T}}\Bigr)}.
\end{align*}
For $T>0$, $\coth(kT)$ is monotonically decreasing in $k\geqslant 1$. Consequently, the denominator of the fraction above is monotonically decreasing, so $\sigma_k^1$ is monotonically increasing with respect to $k$. This is analogous to the monotonicity argument in \cite[Section 3]{fraser2011first}. If $\gamma>0$, the sequence $\{\sigma_k^1\}$ is also monotonically increasing in $k$, following an argument similar to the proof of \cite[Lemma 2.1]{fan2015extremal}.

From the above monotonicity analysis, for any eigenvalue $\sigma>0$, the subspace spanned by its eigenfunctions of positive frequencies is either:
\begin{itemize}
  \item two-dimensional, when $\sigma = \sigma_{k_1}^1 = \sigma_{k_2}^0$ for some $k_1<k_2$; or
  \item four-dimensional, when $\sigma$ belongs to exactly one of the sequences $\{\sigma_k^0\}$ or $\{\sigma_k^1\}$.
\end{itemize}
Thus, the image of $\Phi$, given by Steklov eigenfunctions of \textit{non-zero} frequencies $k$, lies in a $2$-dimensional (trivial case) or a $4$-dimensional subspace.

To simplify the computation, by an appropriate scaling $c>0$ we may assume $H_T=1$. Equation \eqref{eqn:quadratic eq steklov} then reduces to
\begin{align}
\label{eqn: quadratic eq reduced}
\sigma^2 - k\left(1+\frac{1}{H_0}\right)\coth(kT)\,\sigma + \frac{k^2}{H_0}=0.
\end{align}
Clearly, the sign of $H_0$ in \eqref{eqn: quadratic eq reduced} coincides with the sign of $\epsilon$ in the original system \eqref{eq:annulus}. We then divide the proof of Theorem \ref{thm:fbmi cases} into different cases according to the sign of $H_0$.

\subsection{Proof of Theorem \ref{thm:fbmi cases} and Theorem \ref{thm:fbmi cases-2}}
We now prove Theorems~\ref{thm:fbmi cases} and~\ref{thm:fbmi cases-2} together. Since both theorems are organized according to cases depending on the dimension of the image of $\Phi$ in the same order, we unify their proofs and divide the argument into the subcases stated in the theorems.

We first consider the case $\epsilon = -1$. In this situation, $H_0 < 0$, so for each $k \geqslant 1$ equation \eqref{eqn: quadratic eq reduced} has a pair of roots of opposite signs. Consequently, each eigenvalue corresponding to eigenfunctions of non-zero frequency $k$ has multiplicity $2$.

For the case $k=0$, the eigenfunction is of the form $l(t) = c_1 t + c_2$. From the boundary conditions in~\eqref{eq:annulus_H} we obtain
\begin{align*}
-c_1 = \sigma H_0 c_2,\qquad c_1 = \sigma (c_1 T + c_2).
\end{align*}
Clearly, it cannot happen that all coordinate functions of $\Phi$ are affine in $t$ (the $k=0$ case), because then $\Phi$ would not be an immersion. Hence at least one coordinate function must involve eigenfunctions with factors of frequency $k>0$. As explained in the previous section, we restrict to $\sigma > 0$. The boundary conditions then yield
\begin{align}
\label{eqn:linear eigenvalue}
\sigma = \frac{H_0 + 1}{H_0 T},\qquad H_0 < -1.
\end{align}
For $H_0 < -1$, $\sigma(T)$ is decreasing in $T>0$, ranging from $+\infty$ to $0$.

Suppose that for some $k \geqslant 1$ and $H_0 < -1$ there exists $T>0$ such that
\begin{align}
\label{eqn:steklov equal with zero}
\sigma_k^1(H_0,T) = \sigma_0(H_0,T) = \frac{H_0+1}{H_0 T}.
\end{align}
Let $f(t)$ and $l(t)$ be the functions corresponding to frequencies $k$ and $0$, respectively, defined as before. Consider the map
\begin{align}
\label{eqn:fmbi 3-dim opposite}
\Phi(t,\theta)=\begin{pmatrix}
f(t) \cos k\theta\\
f(t)\sin k\theta\\
l(t)
\end{pmatrix}.
\end{align}

In fact, let $\Phi: \mathbb{S}^1\times [0,T] \to \mathbb{E}^{n}$ be a conformal map onto its image with respect to $\mathbb{S}^1$-invariant metrics, whose coordinate functions are spanned by eigenfunctions of frequencies $k \geqslant 1$ and $0$. Then, after an orthogonal transformation of $\mathbb{E}^n$, the image of the conformal map lies in $\mathbb{E}^3$ and takes the form \eqref{eqn:fmbi 3-dim opposite}. This can be shown using the facts that $g(\partial_t,\partial_t)$ and $g(\partial_\theta,\partial_\theta)$ depend only on $t$, and $g(\partial_t,\partial_\theta) = 0$. Recall that $g$ denotes the metric induced from the Euclidean one on $\Sigma$.

From expression \eqref{eqn:fmbi 3-dim opposite}, the image of $\Phi$ is conformal to the flat metric on the annulus if and only if
\begin{align}
k^2(a^2 - b^2) - c_1^2 = 0.
\end{align}
This requires $a^2 - b^2 > 0$. For fixed $H_0 < -1$ and $k \geqslant 1$, both solutions of equation \eqref{eqn: quadratic eq reduced} are decreasing in $T>0$. Passing to the limit $T \to +\infty$ in equation \eqref{eqn: quadratic eq reduced}, we obtain solutions $\left\{ \dfrac{k}{H_0},\, k \right\}$. Hence $\sigma_k^1(H_0,T) > k$ (since it is decreasing in $T$) and $\left| \dfrac{b}{a} \right| > k \cdot \dfrac{|H_0|}{k} > 1$, which implies $a^2 - b^2 < 0$, a contradiction. Therefore, if $\epsilon = -1$, the FBMI can only be induced by eigenfunctions of strictly positive frequencies and the image lies in a two-dimensional subspace.

For each $k \geqslant 1$, $T>0$, and $H_0<0$, let $\sigma$ be the positive solution of \eqref{eqn: quadratic eq reduced}. Setting $a=1$, the boundary conditions \eqref{eqn:boundary condition flat} give $b = -\dfrac{\sigma H_0}{k} > 0$. Define $\Phi$ as
\begin{align}
\label{eqn:2-dim form}
\Phi(t,\theta) = \bigl( f(t)\cos(k\theta),\; f(t)\sin(k\theta) \bigr),
\end{align}
where $f(t) = \cosh(kt) + b\sinh(kt)$. Since $b>0$, one readily checks that $\Phi$ is a conformal immersion into $\mathbb{E}^2$. Thus, for any $k \geqslant 1$ and $T>0$, we obtain an FBMI from $(\Sigma^2,g)$ with $g_{T} \in [g]$ into the annulus bounded by two concentric circles in $\mathbb{E}^2$.

Setting $a=1$, the radii of these concentric circles are
\begin{align*}
r^2 = f^2(0) = 1,\qquad R^2 = f^2(T) = \bigl(\cosh(kT) + b\sinh(kT)\bigr)^2.
\end{align*}
Clearly, for any fixed $T>0$ and $k \geq 1$,
\begin{align*}
\lim_{b \to +\infty} \frac{R}{r} = \infty.
\end{align*}
This completes the proof for the case $\epsilon = -1$.

We now consider the case $\epsilon = 1$, which implies $H_0 > 0$. Note that in this case the coordinate functions of $\Phi$ cannot be induced by a single positive frequency $k \geqslant 1$. Indeed, if that were the case, the image of $\Phi$ would lie in $\mathbb{E}^2$, take the form \eqref{eqn:2-dim form}, and intersect both concentric circles from the interior. Such a $\Phi$ clearly cannot be an immersion.

We first assume that the coordinate functions of $\Phi$ are spanned only by eigenfunctions of positive frequencies. Equation \eqref{eqn: quadratic eq reduced} then yields a pair of positive roots for each $k \geqslant 1$. Using an argument analogous to \cite[Lemma 2.1]{fan2015extremal}, we obtain the following: for any fixed $k_1 \geqslant 1$ and $T > 0$, and for any sufficiently large $k_2 > k_1$, there exists $H_0 > 0$ (depending on $k_2$ and $T$) such that
\begin{align}
\label{eqn:repeated eigenvalue}
\sigma_{k_1}^1(H_0,T) = \sigma_{k_2}^0(H_0,T).
\end{align}
The same statement holds if the roles of $H_0$ and $T$ are interchanged.

Set $f_i(t) = a_i \cosh(k_i t) + b_i \sinh(k_i t)$ for $i = 1,2$, and consider the embedding
\begin{align}
\label{eqn:embedding example 1}
\Phi(t,\theta) = \begin{pmatrix}
f_1(t)\cos(k_1\theta)\\
f_1(t)\sin(k_1\theta)\\
f_2(t)\cos(k_2\theta)\\
f_2(t)\sin(k_2\theta)
\end{pmatrix}.
\end{align}
Suppose $\Phi$ is a conformal map whose coordinate functions are spanned by eigenfunctions of frequencies $k_2 > k_1 > 0$. Then, after an orthogonal transformation of the ambient space $\mathbb{E}^{n}$, it necessarily takes the form \eqref{eqn:embedding example 1}. The proof is entirely analogous to the reasoning following formula \eqref{eqn:fmbi 3-dim opposite}.

From the conformality condition we obtain
\begin{align}
\label{eqn:conformal condition}
k_1^2(a_1^2 - b_1^2) + k_2^2(a_2^2 - b_2^2) = 0.
\end{align}
The map $\Phi$ is degenerate only if $f_1'(t_0) = f_2'(t_0) = 0$ for some $t_0 \in [0,T]$. This gives
\begin{align*}
\frac{b_1}{a_1} = -\tanh(k_1 t_0),\qquad \frac{b_2}{a_2} = -\tanh(k_2 t_0).
\end{align*}
However, the first equation in \eqref{eqn:boundary condition flat} yields
\begin{align*}
\frac{k_1 b_1}{a_1} = \frac{k_2 b_2}{a_2}.
\end{align*}
This is impossible because the function $x \mapsto x \tanh(c x)$ is strictly increasing for $x > 0$ when $c > 0$.

To construct examples satisfying the conformal condition \eqref{eqn:conformal condition}, fix $H_0 > 0$ and $k_1 \geqslant 1$. Note that the eigenvalues $\sigma_{k_1}^1$ are the larger roots of \eqref{eqn: quadratic eq reduced} and decrease monotonically with respect to $T > 0$. Analogously to the case $\epsilon = -1$ treated above, we obtain
\begin{align*}
\sigma_{k_1}^1(H_0,T) > \max\left\{\frac{k_1}{H_0},\, k_1\right\} \geqslant \frac{k_1}{H_0},\qquad \forall T>0.
\end{align*}
Consequently, $\left|\dfrac{b_1}{a_1}\right| = \dfrac{\sigma_{k_1}^1 H_0}{k_1} > 1$.

Similarly, we have $\sigma_{k_2}^0(H_0,T) < \dfrac{k_2}{H_0}$. Indeed, since $\sigma_{k_2}^1(H_0,T)$ is decreasing in $T$ and $\sigma_{k_2}^0(H_0,T)\,\sigma_{k_2}^1(H_0,T) = k_2^2/H_0$, it follows that $\sigma_{k_2}^0(H_0,T)$ is increasing in $T$ and
\begin{align*}
\sigma_{k_2}^0(H_0,T) < \min\left\{\frac{k_2}{H_0},\, k_2\right\} \leqslant \frac{k_2}{H_0},\qquad \forall T>0.
\end{align*}
Thus we always have
\begin{align*}
\left|\frac{b_2}{a_2}\right| = \frac{\sigma_{k_2}^0 H_0}{k_2} < 1,
\end{align*}
which implies $a_2^2 - b_2^2 > 0$.

Hence, fix $k_1, T > 0$. For sufficiently large $k_2$, choose $H_0 > 0$ such that \eqref{eqn:repeated eigenvalue} holds. Then for any pair $(a_1, b_1)$ with $a_1 b_1 \neq 0$, we may select $a_2, b_2$ so that both \eqref{eqn:repeated eigenvalue} and \eqref{eqn:conformal condition} are satisfied.

The radii of the concentric spheres are given by
\begin{align*}
r^2 &= |\Phi(0)|^2 = a_1^2 + a_2^2,\\
R^2 &= |\Phi(T)|^2 = \bigl(a_1\cosh(k_1 T)+b_1\sinh(k_1 T)\bigr)^2 + \bigl(a_2\cosh(k_2 T)+b_2\sinh(k_2 T)\bigr)^2.
\end{align*}
We show that for any $T>0$, one can always find an FBMI of the form \eqref{eqn:embedding example 1} lying between spheres of distinct radii.

First, observe that when computing the ratio $R/r$, we may set $a_1=1$. Fix $T>0$ and $k_1\geqslant 1$, and let $k_2$ be sufficiently large. Define $s_i = \dfrac{\sigma H_0}{k_i}$, where $\sigma = \sigma_{k_1}^1(H_0,T)=\sigma_{k_2}^0(H_0,T)$. Then $s_1,s_2$ satisfy the system
\begin{align}
\label{eqn:4 dim eigenvalue normalized}
\begin{cases}
s_1^2 - (1+H_0)\coth(k_1T)\,s_1 + H_0 = 0,\\[4pt]
s_2^2 - (1+H_0)\coth(k_2T)\,s_2 + H_0 = 0.
\end{cases}
\end{align}
For fixed $k_1$ and $T$, the first quadratic equation in $s_1$ is completely determined by $H_0$; the second depends on $k_2$ and $H_0$. This system is overdetermined: the relation $H_0\sigma = s_1k_1 = s_2k_2$ implies that, given $k_2$, compatibility holds only for special values of $H_0$ (depending on $k_2$). As roots of quadratics, $s_1$ is the larger root of the first equation, while $s_2$ is the smaller root of the second. To find parameters that make the system compatible, we examine the asymptotic behavior of the solutions independently.

First, if we fix $H_0>0$, the second equation shows that $s_2$ is bounded below by a positive constant as $k_2\to\infty$. Consequently, $\sigma = s_2k_2/H_0 \to \infty$ as $k_2\to\infty$. Next, fix $k_i$ and $T$, and let $H_0\to 0^+$. In this limit, the larger roots of both quadratics tend to $\coth(k_iT)>0$, while the smaller roots tend to $0$ at a rate of $O(H_0)$. Now fix $k_1,k_2,T$ and solve for $\sigma$ from the two equations in \eqref{eqn:4 dim eigenvalue normalized} separately. As $H_0\to0^+$, the first equation yields $\sigma = s_1k_1/H_0\to\infty$, whereas the second gives $\sigma = s_2k_2/H_0$ which remains bounded. For fixed $k_1$, $T>0$, and $\delta>0$, and for each sufficiently large $k_2$, there exists some $H_0\in(0,\delta)$ (depending on $k_2$) such that system \eqref{eqn:4 dim eigenvalue normalized} is compatible. Thus, for fixed $k_1$ and $T>0$, and for all sufficiently large $k_2$, a positive parameter $H_0$ (depending on $k_2$) that ensures compatibility exists and can be chosen so that
\begin{align*}
\lim_{k_2\to\infty} H_0 = 0.
\end{align*}

As $H_0 \to 0^+$, since $s_1$ is the larger solution of the first equation in \eqref{eqn:4 dim eigenvalue normalized}, we have
\begin{align*}
\lim_{H_0\to 0^+} s_1 = \coth(k_1 T).
\end{align*}
With $a_1 = 1$, the first equation of \eqref{eqn:boundary condition flat} then implies
\begin{align*}
\lim_{k_2\to\infty} b_1 = \lim_{H_0\to 0^+} (-s_1) = -\coth(k_1 T).
\end{align*}
Since $s_2$ is the smaller solution of the corresponding equation in \eqref{eqn:4 dim eigenvalue normalized}, we obtain (also from \eqref{eqn:boundary condition flat})
\begin{align*}
\lim_{k_2\to\infty} \frac{b_2}{a_2} = \lim_{H_0\to 0^+} (-s_2) = 0.
\end{align*}
Equation \eqref{eqn:conformal condition} then forces $a_2^2$ to decay like $1/k_2^2$ as $k_2\to\infty$. Consequently,
\begin{align*}
\bigl(a_2\cosh(k_2 T)+b_2\sinh(k_2 T)\bigr)^2 = a_2^2\left(\cosh(k_2 T) + \frac{b_2}{a_2}\sinh(k_2 T)\right)^2
\end{align*}
behaves asymptotically as $e^{2k_2 T}/k_2^2$, which tends to $+\infty$ as $k_2\to\infty$. Therefore, for fixed $k_1 \geq 1$ and $T>0$,
\begin{align}
\lim_{k_2\to\infty} \frac{R}{r} = \infty.
\end{align}

Now consider the case $H_0>0$ and suppose that eigenfunctions of frequency $k=0$ (i.e., functions of the form $l(t)=c_1 t+c_2$) appear as factors of some $\phi_i$. Using formula \eqref{eqn:linear eigenvalue}, for fixed $H_0$, the function $\sigma(T)$ decreases from $+\infty$ to $0$ as $T>0$ increases. Moreover, as in \cite[Lemma 2.1]{fan2015extremal}, we have $\sigma_0(T) < \sigma_k^1(T)$ for all $k \geqslant 1$ and $T>0$. For any $T>0$ and $k \geqslant 1$, we can only possibly have $\sigma_0(T) = \sigma_k^0(T)$, and the eigenspace of $\sigma_0(T)$ has dimension at most $3$. Thus, as in the case $H_0<0$, after an orthogonal transformation the ambient space of an FBMI with boundaries on concentric spheres that involves linear eigenfunctions is three-dimensional and takes the form
\begin{align}
\label{eqn:fbmi 3 dimension}
\Phi(t,\theta) = \begin{pmatrix}
f(t)\cos k\theta\\
f(t)\sin k\theta\\
t + c_2
\end{pmatrix},
\end{align}
where $f(t)=a\cosh(kt)+b\sinh(kt)$ is an eigenfunction corresponding to $\sigma_k^0(T)$ with $a > -b \geqslant 0$. The conformality condition gives
\begin{align}
\label{eqn:conformal condition 2}
k^2(a^2 - b^2) - 1 = 0.
\end{align}
Under this condition, $\Phi$ is degenerate only when $\langle \Phi_\theta,\Phi_\theta\rangle = 0$, which occurs when $f(t)=0$. If $f(t_0)=0$, then $\dfrac{b}{a} = -\coth(kt_0) < -1$. However, from \eqref{eqn:conformal condition 2} we have $a^2 = \dfrac{1}{k^2} + b^2$, which yields a contradiction. Hence $\Phi$ is always an immersion.

We now analyze the ratio of the sphere radii $r/R$ in this case. From the boundary conditions,
\begin{align*}
l(0)=c_2 = -\frac{T}{H_0+1},\qquad l(T)=T+c_2 = \frac{H_0 T}{H_0+1},
\end{align*}
we obtain
\begin{align*}
|\Phi(0)|^2 = \frac{T^2}{(H_0+1)^2} + a^2,\qquad 
|\Phi(T)|^2 = \bigl(a\cosh kT + b\sinh kT\bigr)^2 + \frac{H_0^2 T^2}{(H_0+1)^2}.
\end{align*}
Set $s_k = \dfrac{\sigma_k^0 H_0}{k} = \dfrac{H_0+1}{T k}$. We have the system
\begin{align}
\label{eqn:3 dim eigenvalue system}
\begin{cases}
s_k = \dfrac{H_0+1}{T k},\\[6pt]
s_k^2 - (1+H_0)\coth(Tk)\,s_k + H_0 = 0.
\end{cases}
\end{align}
Since $T$ always appears as the product $Tk$, $s_k$ depends on $H_0$ and $Tk$ rather than on $k$ alone. This system is overdetermined: for a given $H_0>0$, compatibility holds only for special values of $Tk$, and those values are the same for all $k \geq 1$. Thus, it suffices to consider $k=1$. For $k=1$ we obtain
\begin{align}
\label{eqn:parameter relation 3-dim}
\frac{1}{T^2} - \frac{\coth T}{T} + \frac{H_0}{(H_0+1)^2} = 0.
\end{align}
One checks that the term $\dfrac{1}{T^2} - \dfrac{\coth T}{T}$ is negative and monotonically increasing in $T>0$, while $\dfrac{H_0}{(H_0+1)^2}$ attains its maximum $1/4$ for $H_0>0$. Moreover,
\begin{align*}
\lim_{T\to +\infty} \left( \frac{1}{T^2} - \frac{\coth T}{T} \right) = 0.
\end{align*}
Hence, for each $H_0>0$, equation \eqref{eqn:parameter relation 3-dim} has a unique solution $T \in [T_0,\infty)$, where $T_0>0$ satisfies
\begin{align}
\label{eqn: T0 value}
\frac{1}{T_0^2} - \frac{\coth T_0}{T_0} + \frac{1}{4} = 0.
\end{align}
We have $\displaystyle\lim_{H_0\to\infty} \frac{T(H_0)}{H_0}=1$ and $\displaystyle\lim_{H_0\to 0} H_0 T(H_0) = 1$. For $k=1$, from \eqref{eqn:3 dim eigenvalue system} and \eqref{eqn:boundary condition flat},
\begin{align*}
\lim_{H_0\to 0} \frac{-b}{a} = \lim_{H_0\to 0} \frac{H_0+1}{T} = 0.
\end{align*}
Thus $a\to 1$ and $b\to 0^-$ as $H_0\to 0$ (see \eqref{eqn:conformal condition 2}). As $H_0\to 0$ we have $T\to\infty$, so $|\Phi(0)|^2$ grows like $T^2$ while $|\Phi(T)|^2$ grows like $e^{2T}$. Consequently,
\begin{align*}
\lim_{H_0\to 0} \frac{|\Phi(0)|}{|\Phi(T)|} = 0.
\end{align*}
Having summarized all the arguments, we complete the proof.

\begin{remark}
\label{remark 1}
For equation \eqref{eqn:parameter relation 3-dim}, for each fixed $H_0>0$ there exists a unique $T \geqslant T_0$ satisfying it (by monotonicity in $T$). Conversely, for any given $T > T_0$, there are two values of $H_0$ solving \eqref{eqn:parameter relation 3-dim} (viewed as a quadratic in $H_0$).

If we replace $T$ by $Tk$ in \eqref{eqn:parameter relation 3-dim} and in the subsequent reasoning, we obtain an analog of the $4$-dimensional case given by \eqref{eqn:embedding example 1}: for any fixed conformal class on the annulus (parametrized by $T>0$), an FBMI into $\mathbb{E}^{n}$ lying in a $3$-dimensional subspace with boundaries on concentric spheres, given by \eqref{eqn:fbmi 3 dimension}, exists for all sufficiently large $k \geq 1$. Moreover, by letting $H_0 \to 0^+$, we can arrange $\displaystyle\lim_{k\to\infty} \frac{R}{r} = \infty$.

We also note that the solution $T_0$ of \eqref{eqn: T0 value} corresponds exactly to the conformal class that defines the critical catenoid in \cite[Section 3]{fraser2011first}, i.e., $\dfrac{T_0}{2} = \coth(T_0/2)$. For $k=1$, the FBMI $\Phi$ given by \eqref{eqn:fbmi 3 dimension} is an embedding and is uniquely determined when $T = T_0$. In this case, the image of $\Phi$ is the critical catenoid cut by a single sphere in $\mathbb{E}^3$, as in \cite{fraser2011first}.
\end{remark}

\section{Index Bounds and Morse Theory}\label{sec:index}

The aim of this section is to demonstrate that the framework for the spectral index, developed in~\cite{medvedev2023index,medvedev2025free}, can be adapted to the setting of \textit{compact} FBMI with boundary in a collection of concentric spheres. This framework also allows us to derive explicit lower and upper bounds for the Morse index. Furthermore, we determine the exact Morse index of a 
$k$-dimensional flat annulus embedded in a spherical shell.

\subsection{Spectral Index}

To establish a variational characterization of the eigenvalues for Problem~\eqref{eq:steklov_weight}, we follow~\cite{karpukhin2023weyl} and introduce the function space
\[
W_\partial(\Sigma,g,\rho) := \left\{ f \in C^\infty(\bar{\Sigma}) : \|f\|_{W^{1,2}_\partial(\Sigma,g,\rho)} < +\infty \right\}, 
\]
where the squared norm is defined as
\[
\|f\|_{W^{1,2}_\partial(\Sigma,g,\rho)}^2 := \int_\Sigma |\nabla^g f|^2 \,dv_g + \int_{\partial\Sigma} f^2 |\rho| \,ds_g.
\]

Next, we define the space $W^{1,2}_\partial(\Sigma,g,\rho)$ as the closure of $W_\partial(\Sigma,g,\rho)$ with respect to the  $\|\cdot\|_{W^{1,2}_\partial(\Sigma,g,\rho)}$-norm. The variational characterization of the \textit{non-zero} eigenvalues $\sigma^\pm_j(\Sigma,g,\rho)$ is then given by
\begin{equation}\label{var}
    \frac{\pm 1}{\sigma^\pm_j(\Sigma,g,\rho)} = \min_{F_j} \max_{u \in F_j \cap X^\pm \setminus \{0\}} \pm \dfrac{\displaystyle\int_{\partial \Sigma} u^2 \rho \,ds_g}{\displaystyle\int_{\Sigma} |\nabla^g u|^2 \,dv_g},
\end{equation}
where $X \subset W^{1,2}_\partial(\Sigma,g,\rho)$ is the subspace of functions orthogonal to the constant function with respect to the weight $\rho$ (i.e., $\displaystyle\int_{\partial\Sigma} u \rho \,ds_g = 0$), $F_j$ is a subspace of $X$ of codimension $j-1$, and we have introduced the subsets
\begin{align*}
    X^+ &:= \left\{ f \in X : \int_{\partial \Sigma} f^2 \rho \,ds_g > 0 \right\}, \\
    X^- &:= \left\{ f \in X : \int_{\partial \Sigma} f^2 \rho \,ds_g < 0 \right\}.
\end{align*}
(Note that $X^\pm$ are cones rather than linear spaces due to the strict inequalities.)

Given the explicit form of $\rho$ in~\eqref{rho}, we can rewrite \eqref{var} as
\begin{equation}\label{var+}
    \frac{1}{\sigma^+_j(\Sigma,g,\rho)} = \min_{F_j} \max_{u \in F_j \cap X^+ \setminus \{0\}} \dfrac{\displaystyle\sum_\alpha \frac{1}{R_\alpha} \int_{\partial_{+}^\alpha \Sigma} u^2 \,ds_g - \sum_\beta \frac{1}{r_\beta} \int_{\partial_{-}^\beta \Sigma} u^2 \,ds_g}{\displaystyle\int_{\Sigma} |\nabla^g u|^2 \,dv_g},
\end{equation}
and
\begin{equation}\label{var-}
    \frac{-1}{\sigma^-_j(\Sigma,g,\rho)} = \min_{F_j} \max_{u \in F_j \cap X^- \setminus \{0\}} \dfrac{-\displaystyle\sum_\alpha \frac{1}{R_\alpha} \int_{\partial_{+}^\alpha \Sigma} u^2 \,ds_g + \sum_\beta \frac{1}{r_\beta} \int_{\partial_{-}^\beta \Sigma} u^2 \,ds_g}{\displaystyle\int_{\Sigma} |\nabla^g u|^2 \,dv_g},
\end{equation}
where the subsets $X^\pm$ take the explicit form:
\begin{align*}
    X^+ &:= \left\{ f \in X : \sum_\alpha \frac{1}{R_\alpha} \int_{\partial_{+}^\alpha \Sigma} f^2 \,ds_g - \sum_\beta \frac{1}{r_\beta} \int_{\partial_{-}^\beta \Sigma} f^2 \,ds_g > 0 \right\}, \\
    X^- &:= \left\{ f \in X : \sum_\alpha \frac{1}{R_\alpha} \int_{\partial_{+}^\alpha \Sigma} f^2 \,ds_g - \sum_\beta \frac{1}{r_\beta} \int_{\partial_{-}^\beta \Sigma} f^2 \,ds_g < 0 \right\}.
\end{align*}

We are now ready to define the spectral index, following the framework established in~\cite{karpukhin2022laplace,medvedev2023index,medvedev2025free}.

\begin{definition}\label{def:ind_spec}
We define the \emph{spectral index} $\Ind_S(\Sigma,g,\rho)$ as the number of non-negative eigenvalues $\sigma^+_j(\Sigma,g,\rho)$ (counted with multiplicity) that are strictly less than $1$. Equivalently, for an FBMI in a collection of concentric balls, the spectral index is the maximal dimension of a subspace $W \subset X$ on which the quadratic form
\begin{equation}\label{eq:quadratic_form}
S_S(u,u) = \int_\Sigma |\nabla^g \widehat{u}|_g^2 \, dv_g - \int_{\partial \Sigma} u^2 \rho \, ds_g
\end{equation}
is negative definite, where $\widehat{u}$ denotes the harmonic extension of $u$ from $\partial\Sigma$ to $\Sigma$.
\end{definition}

\begin{remark}
\begin{enumerate}[(i)]
    \item It follows from the definition of $X^-$ that the quadratic form $S_S$ is positive definite on $X^-$. Similarly, $S_S$ can only take negative values when restricted to $X^+$.
    \item The equivalence of the above two definitions follows immediately from the divergence theorem: if $u$ is a non-zero solution to \eqref{eq:steklov_weight}, then
    $$
    S_S(u,u)=(\sigma-1)\int_{\partial \Sigma} u^2 \rho \, ds_g.
    $$
    As we said above, $u \in X^+$. Therefore, $\displaystyle\int_{\partial \Sigma} u^2 \rho \, ds_g>0$ and $\sigma$ must be less than $1$ in order to have $S_S(u,u)<0$. Then by the variational characterization~\eqref{var}, $\sigma>0$.
\end{enumerate}
\end{remark}

For simplicity, we will denote the spectral index by $\Ind_S(\Sigma)$ when the metric $g$ and weight $\rho$ are clear from the context. 
 
\subsection{Index Upper Bounds} By standard arguments (see~\cite{medvedev2023index,medvedev2025free}), we can prove the following proposition.

\begin{proposition}\label{prop:spec_ind}
Let $\Sigma$ be an $m$-dimensional free boundary minimal submanifold with boundary on a collection of concentric spheres in $\mathbb{E}^{n}$. Then we have
\[
\Ind_E(\Sigma) \leqslant n \Ind_S(\Sigma).
\]
\end{proposition}

\begin{proof} 
Let $W$ be a maximal subspace on which the quadratic form $S_S$ is negative definite, so that $\dim W = \Ind_S(\Sigma)$. Suppose, for the sake of contradiction, that $\Ind_E(\Sigma) > n \Ind_S(\Sigma)$. By a standard dimension-counting argument, there exists a non-trivial vector field $X$ such that $S_E(X,X) < 0$ and each component $X^i$ of $X$ is orthogonal to $W$ (with respect to the relevant boundary inner product). This orthogonality implies that $S_S(X^i, X^i) \geqslant 0$ for all $i = 1, \ldots, n$.

Observe that by the Dirichlet principle (the minimizing property of the harmonic extension operator $f \mapsto \widehat{f}$), we have
\[
\int_{\Sigma} |\nabla^g \widehat{f}|^2 \, dv_g \leqslant \int_{\Sigma} |\nabla^g f|^2 \, dv_g
\]
for any sufficiently regular function $f$. Applying this to each component of $X$, we obtain
\[
\int_{\Sigma} |\nabla^g \widehat{X}^i|^2 \, dv_g \leqslant \int_{\Sigma} |\nabla^g X^i|^2 \, dv_g, \quad \text{for } i = 1, \ldots, n.
\]
Consequently, we have the following chain of inequalities:
\[
0 > S_E(X, X) \geqslant \sum_{i=1}^n S_S(X^i, X^i) \geqslant 0,
\]
which yields a contradiction.
\end{proof}

Furthermore, we apply a result of Lima~\cite{lima2023bounds}, which states that for a compact $2$-dimensional FBMS in \textit{any} ambient Riemannian manifold,
\[
\Ind(\Sigma) \leqslant \Ind_E(\Sigma) + \dim \mathcal{M}(\Sigma),
\]
where $\mathcal{M}(\Sigma)$ denotes the moduli space of conformal classes on $\Sigma$. This yields the upper bound in Theorem~\ref{thm:ind_M}, since a collection of concentric spheres can always be embedded into a smooth bounded domain in Euclidean space. 
 
 \subsection{Index Lower Bounds}
 
 The following theorem is a direct adaptation of Theorem A in~\cite{ambrozio2018index}.  

\begin{theorem}\label{ACS}
Let $M^n$ be a compact, orientable, connected, and properly embedded free boundary minimal hypersurface with boundary on a collection of concentric spheres in $\mathbb{E}^{n+1}$, such that $M$ meets the boundary spheres from the convex side. Then
\[
\Ind(M) \geqslant \frac{2}{n(n+1)} \dim H^1(M,\partial M;\mathbb{R}) \geqslant \frac{2}{n(n+1)}(b-1),
\]
where $H^1(M,\partial M;\mathbb{R})$ denotes the first relative cohomology group with real coefficients, and $b$ is the number of boundary components of $M$. In particular, when $n=2$, we have
\[
\Ind(M) \geqslant \frac{1}{3}(2\gamma + b - 1),
\]
where $\gamma$ is the genus of $M$.
\end{theorem}

One way to see that the proof applies to free boundary minimal hypersurfaces (FBMHs) in a collection of concentric balls is to construct a domain $\Omega \subset \mathbb{R}^{n+1}$ containing the FBMH. This can be achieved by suitably modifying or connecting the boundary spheres in neighborhoods of the intersection points to form a new bounding hypersurface. Since the FBMH meets the original spheres from the convex side by construction, this modified boundary can be chosen to preserve convexity at the points of intersection. Consequently, $\Omega$ naturally serves as the ambient Riemannian manifold with convex boundary required by Theorem A, making Theorem \ref{ACS} a direct corollary. 

Alternatively, one can establish Theorem \ref{ACS} directly, without relying on this geometric construction, by straightforwardly verifying that each step in the original proof of Theorem A remains valid under our specific geometric assumptions.

Theorem~\ref{ACS} demonstrates that the Morse index grows at least linearly with respect to the genus and the number of boundary components. 
 
Furthermore, we can also establish a lower bound in terms of the spectral index.

\begin{proof}[Proof of Theorem \ref{+n}]
Since $\Sigma$ is a free boundary minimal hypersurface with boundary on a collection of concentric spheres in $\mathbb{E}^{n}$, its $n$ coordinate functions $\phi_1, \dots, \phi_{n}$ are Steklov eigenfunctions with eigenvalue $1$. Note that these coordinate functions are linearly independent, provided that $\Sigma$ is not totally geodesic (i.e., not contained in a hyperplane). 

Suppose that $\Ind_S(\Sigma) = s$. By definition, there exist $s$ linearly independent eigenfunctions $u_1, \dots, u_s$ with corresponding eigenvalues $0 < \sigma_i^+ < 1$ for $i = 1,\dots, s$.

Observe that $\phi_1, \dots, \phi_{n}$ and $u_1, \dots, u_s$ all belong to $X^+$, since their corresponding eigenvalues are strictly positive. Moreover, the bilinear form $\langle f, h \rangle_\rho := \displaystyle\int_{\partial\Sigma} f h \rho \, ds_g$ defines a valid inner product on $X^+$. Moreover, it is clear that any eigenspace corresponding to a positive eigenvalue belongs to $X^+$. Then, by applying the Gram--Schmidt process to eigenfunctions corresponding to the same eigenvalue, and noting that eigenfunctions corresponding to distinct eigenvalues are automatically orthogonal with respect to this weighted $L^2$ inner product, we may assume without loss of generality that $u_1, \dots, u_k$ form an orthonormal set in $L^2(\partial\Sigma, \rho \, ds_g)$. 

Furthermore, because the $u_i$ have eigenvalues strictly less than $1$ and the $\phi_j$ have eigenvalue exactly $1$, the sets $\{u_1, \dots, u_s\}$ and $\{\phi_1, \dots, \phi_{n}\}$ are mutually orthogonal in $L^2(\partial\Sigma, \rho \, ds_g)$. 

Consider the subspace $W = \text{span}\{u_1, \dots, u_s, \phi_1, \dots, \phi_{n}\}$. Due to the mutual orthogonality and linear independence established above, we have $\dim W = s + n $. We claim that the index form $S$ \eqref{index_S} is negative definite on $W$. 

Indeed, let $\psi \in W$, so that $\psi =\displaystyle \sum_{i=1}^s \alpha_i u_i + \sum_{j=1}^{n} \beta_j \phi_j$. Since $\Sigma$ is a hypersurface, the index form $S$ evaluated on $\psi$ is given by
\begin{equation}\label{form}
S(\psi,\psi) = -\int_\Sigma (\Delta_g\psi + |B|^2\psi)\psi \, dv_g + \int_{\partial\Sigma} \left( \frac{\partial \psi}{\partial \eta} - \psi\rho \right)\psi \, ds_g,
\end{equation}
where $B$ is the second fundamental form of $\Sigma$.

Clearly, $\Delta_g\psi = 0$ because $\psi$ is a linear combination of harmonic functions. On the boundary $\partial\Sigma$, the normal derivative satisfies
\[
\frac{\partial \psi}{\partial \eta} = \left( \sum_{i=1}^s \alpha_i \sigma_i^+ u_i + \sum_{j=1}^{n} \beta_j \phi_j \right) \rho.
\]
Using the orthogonality of the basis functions in $L^2(\partial\Sigma, \rho \, ds_g)$ and the fact that $\displaystyle\int_{\partial\Sigma} u_i^2 \rho \, ds_g = 1$, we can evaluate the boundary integrals. First,
\begin{equation}\label{form1}
\int_{\partial\Sigma} \frac{\partial \psi}{\partial \eta} \psi \, ds_g = \sum_{i=1}^s \alpha_i^2 \sigma_i^+ + \int_{\partial\Sigma} \left( \sum_{j=1}^{n} \beta_j \phi_j \right)^2 \rho \, ds_g.
\end{equation}
Similarly,
\begin{equation}\label{form2}
\int_{\partial\Sigma} \psi^2 \rho \, ds_g = \sum_{i=1}^s \alpha_i^2 + \int_{\partial\Sigma} \left( \sum_{j=1}^{n} \beta_j \phi_j \right)^2 \rho \, ds_g.
\end{equation}
Subtracting \eqref{form2} from \eqref{form1}, the terms involving $\beta_j$ cancel out perfectly. Substituting this into \eqref{form} yields
\[
S(\psi,\psi) = -\int_\Sigma |B|^2 \psi^2 \, dv_g + \sum_{i=1}^s \alpha_i^2 (\sigma_i^+ - 1).
\]
Since $\sigma_i^+ < 1$ for all $i$, the second term is non-positive. The first term is also non-positive. If $S(\psi, \psi) = 0$, then we must have $\alpha_i = 0$ for all $i$, and $\displaystyle\int_\Sigma |B|^2 \psi^2 \, dv_g = 0$. This implies that $\psi = \displaystyle\sum_{j=1}^{n} \beta_j \phi_j$ vanishes on the support of $|B|^2$. Since $\psi$ is harmonic and $\Sigma$ is not totally geodesic (meaning $|B|^2$ is not identically zero), $\psi$ must vanish identically on $\Sigma$, forcing all $\beta_j = 0$. 

Thus, $S(\psi, \psi) < 0$ for all non-zero $\psi \in W$, proving that $S$ is negative definite on $W$. Therefore,
\[
\Ind(\Sigma) \geqslant \dim W = s + n  = \Ind_S(\Sigma) + n.
\]
\end{proof}

\subsection{Index of a Flat Annulus} \label{flat}

We denote the $m$-dimensional flat annulus with the inner radius $r$ and outer radius $R$ by $\mathbb A^m(r,R)$. Note that the spherical shell, i.e., the region between two concentric spheres, is also a flat annulus. We use this terminology when we assume that the dimension of the annulus coincides with the dimension of the ambient Euclidean space. We denote the outer and inner boundaries of $\mathbb{A}^m(r,R)$ by $\partial_{\mathrm{out}}\mathbb{A}^m(r,R)$ and $\partial_{\mathrm{in}}\mathbb{A}^m(r,R)$, respectively.

\begin{lemma}\label{sigma_1}
The coordinate functions of the annulus \(\mathbb{A}^{m}(r,R) \subset \mathbb{A}^n(r,R) \subset \mathbb{E}^n\) are $\sigma^+_1$-eigenfunctions. 
\end{lemma}

\begin{proof}
The embedding of the annulus $\Phi : [r, R] \times \mathbb{S}^{m-1} \to \mathbb{A}^m(r,R)$ is given in spherical coordinates by:
\begin{align*}
\Phi_1(t, \theta_1, \ldots, \theta_{m-1}) &= t \cos \theta_1, \\
\Phi_2(t, \theta_1, \ldots, \theta_{m-1}) &= t \sin \theta_1 \cos \theta_2, \\
&\;\;\vdots \\
\Phi_{m-1}(t, \theta_1, \ldots, \theta_{m-1}) &= t \sin \theta_1 \sin \theta_2 \cdots \sin \theta_{m-2} \cos \theta_{m-1}, \\
\Phi_{m}(t, \theta_1, \ldots, \theta_{m-1}) &= t \sin \theta_1 \sin \theta_2 \cdots \sin \theta_{m-2} \sin \theta_{m-1},
\end{align*}
where $\theta_i \in [0, \pi]$ for $i = 1, \ldots, m - 2$ and $\theta_{m-1} \in [0, 2\pi)$. The Laplacian of the induced metric takes the form:
\[
\Delta_g f = \frac{\partial^2 f}{\partial t^2} + \frac{m - 1}{t} \frac{\partial f}{\partial t} + \frac{1}{t^2} \Delta_{\mathbb{S}^{m-1}} f.
\]
We aim to prove that the coordinate functions $\Phi_1, \dots, \Phi_m$ are eigenfunctions corresponding to the first positive Steklov eigenvalue $\sigma_1^+$.

Since the Laplace--Beltrami eigenfunctions $\{\phi_i\}_{i=0}^{\infty}$ of the standard sphere $\mathbb{S}^{m-1}$ (the spherical harmonics) form an $L^2(\mathbb{S}^{m-1})$-orthonormal basis, any function $f \in L^2(\mathbb{A}^m(r,R))$ can be decomposed as $f(t,x) = \displaystyle\sum_{i=0}^{\infty} a_i(t)\phi_i(x)$, where $t \in [r,R]$ and $x \in \mathbb{S}^{m-1}$. Suppose that $f$ is a $\sigma_1^+$-eigenfunction. The harmonicity condition $\Delta_g f = 0$ implies that each radial coefficient $a_i(t)$ satisfies the ODE:
\begin{equation}\label{eq:rR}
a_{i}^{\prime\prime}(t) + \frac{m - 1}{t}a_{i}^{\prime}(t) + \frac{\lambda_{i}}{t^2}a_{i}(t) = 0, \quad t \in [r,R],
\end{equation}
where $\lambda_i$ is the eigenvalue of $\phi_i$ for the operator $\Delta_{\mathbb{S}^{m-1}}$. The boundary conditions yield:
\begin{equation}\label{eq:brR}
\begin{cases}
    a_{i}^{\prime}(r) = \dfrac{\sigma_{1}^+}{r}a_{i}(r), \\[8pt]
    a_{i}^{\prime}(R) = \dfrac{\sigma_{1}^+}{R}a_{i}(R).
\end{cases}
\end{equation}

A direct computation shows that the weighted boundary integral $\displaystyle\int_{\partial\mathbb{A}^m(r,R)}\rho\,ds_g$, where $\rho$ is given by
\[
\rho = \begin{cases}
\dfrac{1}{R} & \text{on } \partial_{out}\mathbb{A}^m(r,R), \\[6pt]
-\dfrac{1}{r} & \text{on } \partial_{in}\mathbb{A}^m(r,R),
\end{cases}
\] 
vanishes when $m=2$ and is strictly positive when $m>2$. These conditions satisfy the hypotheses of \cite[Theorem 1.2, cases (2) and (3)]{torne2005steklov}, which guarantees that any eigenfunction corresponding to $\sigma_1^+$ possesses exactly two nodal domains. It is a well-known fact that the only spherical harmonics possessing exactly two nodal domains are those in the $\lambda_1$-eigenspace (see \cite{bateman1953higher}). Consequently, any eigenfunction $f$ with two nodal domains can only have non-zero projections onto the $\lambda_0$ and $\lambda_1$ eigenspaces. Note that the $\lambda_0$-eigenspace consists only of constant functions (with $\lambda_0 = 0$), and the multiplicity of the $\lambda_1$-eigenspace (with $\lambda_1 = -(m-1)$) is $m$.

For $m > 2$, the general solutions to \eqref{eq:rR} for these eigenspaces are:
\[
\begin{aligned}
a_{0}(t) &= C_1 + C_2 \frac{t^{2-m}}{2-m}, && \text{for } \lambda_{0} = 0, \\
a_{1, i}(t) &= C_1 t + C_2 t^{1-m}, && \text{for } \lambda_{1} = -(m-1).
\end{aligned}
\]
Differentiating with respect to $t$ and applying the boundary conditions \eqref{eq:brR} at $t \in \{r, R\}$, we obtain for the $\lambda_0$ case:
\begin{equation}\label{sol_1}
C_2 t^{1-m} = \sigma_1^+ \left( \frac{C_1}{t} + C_2 \frac{t^{1-m}}{2-m} \right).
\end{equation}
Multiplying by $t$ and rearranging yields:
\[
C_1 \sigma_1^+ + C_2 t^{2-m} \left( \frac{\sigma_1^+}{2-m} - 1 \right) = 0.
\]
For this to hold at both $t=r$ and $t=R$ (with $r < R$), the coefficients of the linearly independent functions of $t$ must vanish. Since $\sigma_1^+ > 0$, the first term implies $C_1 = 0$. The second term then implies $C_2 = 0$ (as $\dfrac{\sigma_1^+}{2-m} - 1 \neq 0$ for $m > 2$). Thus, $a_0(t) \equiv 0$.

For the $\lambda_1$ case, applying the boundary conditions gives:
\begin{equation}\label{sol_2}
C_1 + C_2 (1-m) t^{-m} = \sigma_1^+ \left( C_1 + C_2 t^{-m} \right).
\end{equation}
Rearranging this yields:
\[
C_1 (1 - \sigma_1^+) + C_2 t^{-m} (1 - m - \sigma_1^+) = 0.
\]
Again, for this to hold at both $t=r$ and $t=R$, we must have $C_1 (1 - \sigma_1^+) = 0$ and $C_2 (1 - m - \sigma_1^+) = 0$. Since we seek a non-trivial solution and $\sigma_1^+ > 0$, the second equation forces $C_2 = 0$ (because $1 - m - \sigma_1^+ < 0$). The first equation then requires $\sigma_1^+ = 1$ (since $C_1 \neq 0$). 

Thus, the only non-trivial solutions are of the form $f(t, x) = C_1 t \varphi(x)$, where $\varphi$ is an eigenfunction in the $\lambda_1$-eigenspace of $\mathbb{S}^{m-1}$. Since the coordinate functions $\Phi_1, \dots, \Phi_m$ are precisely of this form, they are indeed $\sigma_1^+$-eigenfunctions.

For $m = 2$, the solutions to \eqref{eq:rR} take a slightly different form:
\[
\begin{aligned}
a_{0}(t) &= C_1 + C_2 \log t, && \text{for } \lambda_{0} = 0, \\
a_{1, i}(t) &= C_1 t + C_2 t^{-1}, && \text{for } \lambda_{1} = -1.
\end{aligned}
\]
Applying the boundary conditions to $a_0(t)$ yields $C_2 = \sigma_1^+ C_1 + \sigma_1^+ C_2 \log t$. For this to hold at $t=r$ and $t=R$, we must have $\sigma_1^+ C_2 = 0$, which implies $C_2 = 0$ and consequently $C_1 = 0$. Thus, $a_0(t) \equiv 0$.
For $a_{1, i}(t)$, the boundary conditions give:
\[
C_1 - C_2 t^{-2} = \sigma_1^+ \left( C_1 + C_2 t^{-2} \right) \implies C_1 (1 - \sigma_1^+) - C_2 t^{-2} (1 + \sigma_1^+) = 0.
\]
As before, this requires $C_1 (1 - \sigma_1^+) = 0$ and $C_2 (1 + \sigma_1^+) = 0$. Since $\sigma_1^+ > 0$, we must have $C_2 = 0$ and $\sigma_1^+ = 1$. 

Therefore, in all dimensions $m \geqslant 2$, the coordinate functions $\Phi_1, \dots, \Phi_m$ are $\sigma_1^+$-eigenfunctions, which completes the proof.
\end{proof}

\begin{proposition}\label{n-k_A}
The Morse index of an $m$-dimensional flat annulus $\mathbb{A}^m(r,R)$, considered as a free boundary minimal submanifold in an $n$-dimensional spherical shell $\mathbb{A}^n(r,R)$, is at most $n-m$.\end{proposition}

\begin{proof}
Notice that the normal space to $\mathbb{A}^m(r,R)$ at any point $p$ is given by $N_p\mathbb{A}^m(r,R) = \text{span}\{(\partial_1)_{|p}, \dots, (\partial_{n-m})_{|p}\}$. Consequently, any normal vector field $X$ along $\mathbb{A}^m(r,R)$ can be expressed as $X =\displaystyle \sum_{i=1}^{n-m} X^i \partial_i$ for some smooth functions $X^i$.

Let $\{e_1, \dots, e_m\}$ be a local orthonormal frame for $T\mathbb{A}^m(r,R)$. Since the ambient space is Euclidean, the coordinate vector fields $\partial_i$ are parallel, meaning $\nabla_{e_j} \partial_i = 0$. Therefore, the normal connection acts as:
\[
(\nabla_{e_j} X)^\perp = \left( \sum_{i=1}^{n-m} e_j(X^i) \partial_i + X^i \nabla_{e_j} \partial_i \right)^\perp = \sum_{i=1}^{n-m} e_j(X^i) \partial_i.
\]
Consequently, the squared norm of the normal connection is
\[
|\nabla^\perp X|^2 = \sum_{j=1}^m |(\nabla_{e_j} X)^\perp|^2 = \sum_{j=1}^m \sum_{i=1}^{n-m} (e_j(X^i))^2 = \sum_{i=1}^{n-m} |\nabla^g X^i|_g^2.
\]

Since the annulus is totally geodesic, the Simons operator vanishes ($\mathcal{B} \equiv 0$). Thus, the formula for the second variation of the volume functional~\eqref{index_S} yields:
\begin{equation*}
    S(X, X) = \sum_{i=1}^{n-m} \left( \int_{\mathbb{A}^m(r,R)} |\nabla^g X^i|_g^2 \, dv_g - \frac{1}{R}\int_{\partial_{out} \mathbb{A}^m(r,R)} (X^i)^2 \, ds_g + \frac{1}{r}\int_{\partial_{in} \mathbb{A}^m(r,R)} (X^i)^2 \, ds_g \right).
\end{equation*}
By the minimization property of the Dirichlet energy, we obtain:
\[
S(X, X) \geqslant \sum_{i=1}^{n-m} S_S(X^i, X^i).
\]
Consequently, we obtain the bound:
\[
\Ind(\mathbb{A}^m(r,R)) \leqslant (n-m) \Ind_S(\mathbb{A}^m(r,R)).
\]
By Lemma~\ref{sigma_1}, we know that the first positive Steklov eigenvalue is $\sigma_1^+(\mathbb{A}^m(r,R)) = 1$. Since $\sigma_0 = 0$ is the only eigenvalue strictly less than $1$ (and it has multiplicity $1$), we have $\Ind_S(\mathbb{A}^m(r,R)) = 1$. 

Substituting this into our bound yields $\Ind(\mathbb{A}^m(r,R)) \leqslant n-m$, which gives the desired result.
\end{proof}

\begin{proposition}\label{k>2}
Let $m \geqslant 3$. The Morse index of an $m$-dimensional flat annulus $\mathbb{A}^m(r,R)$, considered as a free boundary minimal submanifold in an $n$-dimensional spherical shell $\mathbb{A}^n(r,R)$, is exactly $n-m$.\end{proposition}

\begin{proof}
Applying the second variation formula for the volume functional \eqref{index_S} to the flat annulus $\mathbb{A}^m(r,R)$ yields
\[
S(V,V) = \int_{\mathbb{A}^m(r,R)} |\nabla^\perp V|^2 \, dv_g - \frac{1}{R} \int_{\partial_{out} \mathbb{A}^m(r,R)} |V|^2 \, ds_g + \frac{1}{r} \int_{\partial_{in} \mathbb{A}^m(r,R)} |V|^2 \, ds_g.
\]
Without loss of generality, we may assume the annulus lies in the coordinate subspace $x_1 = \dots = x_{n-m} = 0$ of the ambient Euclidean space. Consequently, the constant coordinate vector fields $\partial_1, \dots, \partial_{n-m}$ form a global orthonormal frame for the normal bundle of $\mathbb{A}^m$. 

Evaluating the index form on these constant normal vector fields $V = \partial_i$ for $i = 1, \dots, n-m$, we observe that $\nabla^\perp \partial_i = 0$ and $|\partial_i|^2 = 1$. Thus, we obtain
\[
S(\partial_i, \partial_i) = -\frac{1}{R} \text{Vol}(\partial_{out} \mathbb{A}^m) + \frac{1}{r} \text{Vol}(\partial_{in} \mathbb{A}^m).
\]
Since the boundary components are spheres of radius $R$ and $r$ respectively, their $(m-1)$-dimensional volumes are $|\mathbb{S}^{m-1}| R^{m-1}$ and $|\mathbb{S}^{m-1}| r^{m-1}$, where $|\mathbb{S}^{m-1}|$ denotes the volume of the unit $(m-1)$-sphere. Substituting these in, we get
\begin{align}\label{S<0}
S(\partial_i, \partial_i) = -\frac{1}{R} |\mathbb{S}^{m-1}| R^{m-1} + \frac{1}{r} |\mathbb{S}^{m-1}| r^{m-1} = |\mathbb{S}^{m-1}| \left( r^{m-2} - R^{m-2} \right).
\end{align}
For $m > 2$ and $r < R$, we have $r^{m-2} < R^{m-2}$, which implies $S(\partial_i, \partial_i) < 0$ for each $i = 1, \dots, n-m$. 

Since the vector fields $\partial_1, \dots, \partial_{n-m}$ are linearly independent, the index form is negative definite on an $(n-m)$-dimensional subspace of normal variations. Hence, the Morse index satisfies $\Ind(\mathbb{A}^m) \geqslant n-m$. Combined with the upper bound established in Proposition~\ref{n-k_A}, we conclude that $\Ind(\mathbb{A}^m) = n-m$, as desired.
\end{proof}

Finally, we consider the case where $m=2$.

\begin{proposition}\label{k=2}
 The Morse index of a $2$-dimensional flat annulus $\mathbb{A}^2(r,R)$, considered as a free boundary minimal submanifold in a $3$-dimensional spherical shell $\mathbb{A}^3(r,R)$, is exactly $1$.
 \end{proposition}

\begin{proof}
By Proposition~\ref{n-k_A}, we know that the index is at most 1, so we only need to construct one normal vector field for which the second variation~\eqref{index_S} is negative. Since the normal bundle is one-dimensional, $S$ takes the form
\[
S(\varphi,\varphi) = \int_{\mathbb A^2(r,R)} |\nabla^g\varphi|^2 \, dv_g 
- \frac{1}{R}\int_{\partial_{out} \mathbb A^2(r,R)} \varphi^2 \, ds_g 
+ \frac{1}{r}\int_{\partial_{in} \mathbb A^2(r,R)} \varphi^2 \, ds_g.
\]

We consider the function 
\[
\varphi(t,\theta) = \ln\frac{t}{r}
\]
in polar coordinates $(t,\theta)$. One has
\[
|\nabla^g\varphi|^2 = \left(\frac{\partial\varphi}{\partial t}\right)^2 + \frac{1}{t^2}\left(\frac{\partial\varphi}{\partial\theta}\right)^2= \frac{1}{t^2}.
\]

 Therefore
\[
\int_{\mathbb A^2(r,R)} |\nabla^g\varphi|^2 \, dv_g 
= \int_{r}^{R}\int_{0}^{2\pi} \frac{1}{t^2}\,t\,d\theta\,dt
= 2\pi \int_{r}^{R} \frac{dt}{t}
= 2\pi \ln\frac{R}{r}.
\]

On the outer boundary \(t = R\), \(\varphi(R) = \ln(R/r)\).  Thus
\[
\frac{1}{R}\int_{\partial_{out} \mathbb A^2(r,R)} \varphi^2 \, ds
= \frac{1}{R} \int_{0}^{2\pi} \left(\ln\frac{R}{r}\right)^2 R\,d\theta
= 2\pi \left(\ln\frac{R}{r}\right)^2.
\]

On the inner boundary \(t = r\), \(\varphi(r) = \ln(r/r) = 0\). Hence
\[
\frac{1}{r}\int_{\partial_{in} \mathbb A^2(r,R)} \varphi^2 \, ds
 = 0.
\]

Substituting into \(S\):
\[
S(\varphi,\varphi) = 2\pi \ln\frac{R}{r} - 2\pi \left(\ln\frac{R}{r}\right)^2
= 2\pi \ln\frac{R}{r}\left(1 - \ln\frac{R}{r}\right)<0,
\]
since $R>r$.
\end{proof}

Propositions \ref{k>2} and \ref{k=2} imply Theorem~\ref{thm:annulus}.

\begin{remark}
    As it follows from \eqref{S<0}, when $m=2$ one has $S(\partial_1,\partial_1)=0$ for the annulus lying in the plane $x_1=0$, i.e., the \textit{nullity} of a flat annulus $\mathbb A^2(r,R)$ in a spherical shell $\mathbb A^2(r,R)$ is at least one.
\end{remark}

\section{Appendix}\label{appendix}

\subsection{On the Index of the Unit Ball}

It is well known that the Morse index of the unit ball $\mathbb{B}^m$, considered as a free boundary minimal submanifold in $\mathbb{B}^n$, is exactly $n-m$ (see, for example,~\cite[Chapter 1, Section 1.5.3]{fraser2020geometric} for the case $m=n-1$). In this section, we present a detailed proof of this fact utilizing the spectral index approach (see~\cite{medvedev2023index} for the definition of the spectral index in this context).  

\begin{lemma}\label{sigma_ball}
The coordinate functions of the $m$-dimensional ball $\mathbb{B}^m \subset \mathbb{B}^n \subset \mathbb{E}^n$ are eigenfunctions corresponding to the first positive Steklov eigenvalue $\sigma_1(\mathbb{B}^m)$, i.e., they satisfy the Steklov problem:
\[
\begin{cases}
    \Delta_g u = 0 & \text{in } \mathbb{B}^m, \\[6pt]
    \dfrac{\partial u}{\partial \eta} = \sigma_1(\mathbb{B}^m) u & \text{on } \partial \mathbb{B}^m,
\end{cases}
\]
where $\sigma_1(\mathbb{B}^m)$ denotes the first positive Steklov eigenvalue.
\end{lemma}

\begin{proof}
The coordinate embedding of the ball $\Phi : [0, 1] \times \mathbb{S}^{m-1} \to \mathbb{B}^m$ is given in spherical coordinates by:
\begin{align*}
\Phi_1(t, \theta_1, \ldots, \theta_{m-1}) &= t \cos \theta_1, \\
\Phi_2(t, \theta_1, \ldots, \theta_{m-1}) &= t \sin \theta_1 \cos \theta_2, \\
&\;\;\vdots \\
\Phi_{m-1}(t, \theta_1, \ldots, \theta_{m-1}) &= t \sin \theta_1 \sin \theta_2 \cdots \sin \theta_{m-2} \cos \theta_{m-1}, \\
\Phi_{m}(t, \theta_1, \ldots, \theta_{m-1}) &= t \sin \theta_1 \sin \theta_2 \cdots \sin \theta_{m-2} \sin \theta_{m-1},
\end{align*}
where $\theta_i \in [0, \pi]$ for $i = 1, \ldots, m - 2$ and $\theta_{m-1} \in [0, 2\pi)$. The Laplace--Beltrami operator of the induced metric takes the form:
\[
\Delta_g f = \frac{\partial^2 f}{\partial t^2} + \frac{m - 1}{t} \frac{\partial f}{\partial t} + \frac{1}{t^2} \Delta_{\mathbb{S}^{m-1}} f.
\]

Since the Laplace--Beltrami eigenfunctions $\{\phi_i\}_{i=0}^{\infty}$ of the standard sphere $\mathbb{S}^{m-1}$ form an $L^2(\mathbb{S}^{m-1})$-orthonormal basis, any function $f \in L^2(\mathbb{B}^m)$ can be decomposed as $f(t,x) = \displaystyle\sum_{i=0}^{\infty} a_i(t)\phi_i(x)$, where $t \in [0,1]$ and $x \in \mathbb{S}^{m-1}$. 

Suppose that $f$ is an eigenfunction corresponding to the first Steklov eigenvalue $\sigma_1(\mathbb{B}^m)$. Then $f$ satisfies the boundary value problem:
\begin{equation}
\begin{cases} 
\Delta_{g}f = 0 & \text{in } \mathbb{B}^{m}, \\[6pt]
\dfrac{\partial f}{\partial \eta} = \sigma_{1}(\mathbb{B}^{m})f & \text{on } \partial \mathbb{B}^{m},
\end{cases}
\end{equation}
where $\eta$ is the outward unit normal. Since $\dfrac{\partial}{\partial \eta} = \dfrac{\partial}{\partial t}$ on the boundary $t=1$, substituting the series expansion into the harmonicity condition $\Delta_g f = 0$ yields the following ordinary differential equation for each radial coefficient $a_i(t)$:
\begin{equation}
\begin{cases} 
a_{i}^{\prime\prime}(t) + \dfrac{m-1}{t}a_{i}^{\prime}(t) + \dfrac{\lambda_{i}}{t^2}a_{i}(t) = 0, & t \in (0,1], \\[8pt]
a_{i}^{\prime}(1) = \sigma_{1}(\mathbb{B}^{m})a_{i}(1),
\end{cases}
\end{equation}
where $\lambda_i$ is the eigenvalue of $\phi_i$ for the operator $\Delta_{\mathbb{S}^{m-1}}$.

By the Courant nodal domain theorem, any eigenfunction corresponding to the first positive Steklov eigenvalue $\sigma_1(\mathbb{B}^m)$ has exactly two nodal domains. It is a well-known fact that the only spherical harmonics possessing exactly two nodal domains are those in the $\lambda_1$-eigenspace (see \cite{bateman1953higher}). Consequently, any such eigenfunction can only have non-zero projections onto the $\lambda_0$ and $\lambda_1$ eigenspaces. Recall that the $\lambda_0$-eigenspace consists solely of constant functions, while the $\lambda_1$-eigenspace (with eigenvalue $\lambda_1 = -(m-1)$ under our sign convention) has multiplicity $m$.

First, consider the radial equation for $\lambda_0 = 0$:
\begin{equation*}
a_{0}^{\prime\prime}(t) + \frac{m-1}{t}a_{0}^{\prime}(t) = 0, \quad t \in (0,1].
\end{equation*}
The general solution is given by:
\[
\begin{aligned}
a_0(t) &= C_1 + C_2\frac{t^{2-m}}{2-m}, && \text{for } m > 2, \\
a_0(t) &= C_1 + C_2 \log(t), && \text{for } m = 2.
\end{aligned}
\]
For the solution to be smooth at the origin $t=0$, we must have $C_2 = 0$. This leaves $a_0(t) = C_1$, which corresponds to a constant function. However, a constant function has only one nodal domain, which contradicts the requirement that a $\sigma_1$-eigenfunction must have exactly two. Thus, the $\lambda_0$ mode does not yield the first positive eigenfunction.

Next, consider the case $\lambda_1 = -(m-1)$, which yields the equation:
\begin{equation*}
a_{1i}^{\prime\prime}(t) + \frac{m-1}{t}a_{1i}^{\prime}(t) - \frac{m-1}{t^2}a_{1i}(t) = 0.
\end{equation*}
We observe that $a_{1i}(t) = t$ is a solution. To find the general solution, we use the substitution $a_{1i}(t) = t g(t)$, which reduces the equation to:
\begin{equation*}
t g^{\prime\prime}(t) + \frac{m+1}{t}g^{\prime}(t) = 0. 
\end{equation*}
Solving this yields $g(t) = C_1 + C_2 \dfrac{t^{-m}}{m}$. Again, the requirement of smoothness at $t=0$ forces $C_2 = 0$, leaving only the constant solution $g(t) = C_1$.

Therefore, the only smooth, non-trivial solutions are of the form $a_{1i}(t) = C_1 t$. This implies that the eigenfunctions are precisely of the form $f(t, x) = t \varphi_l(x)$, where $\varphi_l$ belongs to the $\lambda_1$-eigenspace of $\mathbb{S}^{m-1}$. Recognizing that the coordinate functions $\Phi_i$ of the ball are exactly of this form (with $C_1=1$ and $\varphi_l$ being the standard basis of the $\lambda_1$-eigenspace), we conclude that the $\Phi_i$ are indeed $\sigma_1$-eigenfunctions, and the corresponding eigenvalue is $\sigma_1(\mathbb{B}^m) = 1$.
\end{proof}

\begin{lemma}
\label{n-k}
The index of a unit ball of $m$ dimensions in a unit ball of $n$ dimensions in $\mathbb{E}^n$ is $n-m$.
\end{lemma}

\begin{proof}

The second variation of the volume functional for a normal vector field $X$ takes the form:
\begin{equation}
S(X,X) = \int_{\mathbb{B}^m} |\nabla^\perp X|_g^2 \, dv_g - \int_{\partial \mathbb{B}^m} |X|^2 \, ds_g.
\end{equation}

Assume without loss of generality that $\mathbb{B}^m$ lies in the coordinate plane $\Pi = \{x \in \mathbb{E}^n \mid x_1 = \dots = x_{n-m} = 0\}$, so that $\mathbb{B}^m = \mathbb{B}^n \cap \Pi$. The constant normal vector fields $\partial_1, \dots, \partial_{n-m}$ provide volume-decreasing variations. Consequently, the index form $S$ is negative definite on the $(n-m)$-dimensional subspace $\operatorname{span}\{\partial_1, \dots, \partial_{n-m}\}$. This immediately yields the lower bound $\Ind(\mathbb{B}^m) \geqslant n-m$.

To establish the reverse inequality, observe that the normal space at any point $p \in \mathbb{B}^m$ is $N_p\mathbb{B}^m = \operatorname{span}\{(\partial_1)|_p, \dots, (\partial_{n-m})|_p\}$. Thus, any normal vector field $X$ along $\mathbb{B}^m$ can be written as $X =\displaystyle \sum_{i=1}^{n-m} X^i \partial_i$ for some smooth functions $X^i$. Substituting this into the second variation formula, and noting that $|\nabla^\perp X|^2 = \displaystyle\sum_{i=1}^{n-m} |\nabla^g X^i|_g^2$ and $|X|^2 = \displaystyle\sum_{i=1}^{n-m} (X^i)^2$, we obtain the exact decomposition:
\begin{equation}
\label{final_area}
S(X, X) = \sum_{i=1}^{n-m} \left( \int_{\mathbb{B}^m} |\nabla^g X^i|_g^2 \, dv_g - \int_{\partial \mathbb{B}^m} (X^i)^2 \, ds_g \right).
\end{equation}
By the minimization property of the Dirichlet energy, one has
$$
S(X, X) \geqslant \sum_{i=1}^{n-m} S_S(X^i, X^i),
$$ 
which implies the index bound:
\[
\Ind(\mathbb{B}^m) \leqslant (n-m) \Ind_S(\mathbb{B}^m).
\]

By Lemma \ref{sigma_ball}, the first positive Steklov eigenvalue is $\sigma_1(\mathbb{B}^m) = 1$. According to the definition in~\cite{medvedev2023index}, the spectral index counts the number of eigenvalues strictly less than $1$. Since $\sigma_0 = 0$ is the only such eigenvalue (with multiplicity $1$), we have $\Ind_S(\mathbb{B}^m) = 1$. Substituting this into the bound yields $\Ind(\mathbb{B}^m) \leqslant n-m$. 

Combining this with the lower bound, we conclude that $\Ind(\mathbb{B}^m) = n-m$, which is the desired result.

\end{proof}

\subsection{Geodesic Balls in $\mathbb S^n_+$ and $\mathbb H^n$}
The case of geodesic balls in an 
$n$-dimensional upper hemisphere $\mathbb S^n_+$ or hyperbolic space $\mathbb H^n$ was recently considered in \cite{medvedev2025free}. In Proposition 5.6 of the aforementioned paper, it was proved that the spectral index of an $m$-dimensional geodesic ball as a free boundary minimal submanifold of an $n$-dimensional geodesic ball with the same center is one. While the proof closely follows the arguments presented in the current work, we note a technical inaccuracy in the solution of the associated second-order ordinary differential equations. Importantly, this oversight does not affect the validity of the final result, as the proof in \cite[Proposition 5.6]{medvedev2025free} relies solely on the presence of singularities in the solutions rather than their precise explicit form. We explain it in this section. 

Recall that the upper hemisphere $\mathbb{S}^n$ and the hyperbolic space $\mathbb{H}^n$ are realized as subsets of $\mathbb{R}^{n+1}$ with Cartesian coordinates $x_0, x_1, \ldots, x_n$ and are defined by the equations
\[
x_0^2 + \dots + x_n^2 = 1,\ x_0>0 \qquad -x_0^2 + x_1^2 + \dots + x_n^2 = -1, \ x_0 > 0,
\]
respectively. The geodesic ball in which the free boundary minimal submanifold is embedded has radius $r$ and is centered at the point $(1, 0, \ldots, 0)$. Then the coordinate functions of an $m$-dimensional free boundary minimal submanifold $\Sigma$ satisfy the following~\textit{Robin problem}:

\begin{align*}
\text{In } \mathbb{S}^n_+: \qquad &
\begin{cases}
\Delta_g \phi_i = m \phi_i & \text{in } \Sigma, \ i=0,1,\ldots,n, \\[6pt]
\dfrac{\partial \phi_0}{\partial\eta} = -(\tan r)\phi_0 & \text{on } \partial\Sigma, \\[6pt]
\dfrac{\partial \phi_i}{\partial\eta} = (\cot r)\phi_i & \text{on } \partial\Sigma, \ i=1,\ldots,n,
\end{cases}
\\[12pt]
\text{In } \mathbb{H}^n: \qquad &
\begin{cases}
\Delta_g \phi_i = -m \phi_i & \text{in } \Sigma, \ i=0,1,\ldots,n, \\[6pt]
\dfrac{\partial \phi_0}{\partial\eta} = (\tanh r)\phi_0 & \text{on } \partial\Sigma, \\[6pt]
\dfrac{\partial \phi_i}{\partial\eta} = (\coth r)\phi_i & \text{on } \partial\Sigma, \ i=1,\ldots,n.
\end{cases}
\end{align*}

Consider the case of the hemisphere $\mathbb{S}^n_+$. In \cite[Proposition 5.6]{medvedev2025free}, two second-order ordinary differential equations are derived for the radial components of the $\sigma_0$- and $\sigma_1$-eigenfunctions, respectively (obtained analogously to Lemmas~\ref{sigma_1} and~\ref{sigma_ball}):
\begin{align}
    a_0''(t) + (m-1)(\cot t)\,a_0'(t) + m\,a_0(t) &= 0, \label{eq:a0} \\
    a_{1i}''(t) + (m-1)(\cot t)\,a_{1i}'(t) + \left(m - \frac{m-1}{\sin^2 t}\right)a_{1i}(t) &= 0. \label{eq:a1i}
\end{align}

Applying the substitutions $a_0(t) = z_0(t)\cos t$ and $a_{1i}(t) = z_{1i}(t)\sin t$, we obtain the following reduced equations:
\begin{equation}
\label{eq:z0}
    (\cos t) z_0''(t) + \left( -2 \sin t + \frac{(m-1) \cos^2 t}{\sin t} \right) z_0'(t) = 0,
\end{equation}
and
\begin{equation}
\label{eq:z1i}
    (\sin t) z_{1i}''(t) + (m+1) \cos t \, z_{1i}'(t) = 0.
\end{equation}

The generalized hypergeometric function is defined by the series:
\begin{align*}
{}_pF_q(a_1,\ldots,a_p; b_1,\ldots,b_q; z) &= \sum_{n=0}^{\infty} \frac{(a_1)_n \cdots (a_p)_n}{(b_1)_n \cdots (b_q)_n} \frac{z^n}{n!} \\
&= 1 + \frac{a_1 \cdots a_p}{b_1 \cdots b_q} z + \frac{a_1(a_1+1) \cdots a_p(a_p+1)}{b_1(b_1+1) \cdots b_q(b_q+1)} \frac{z^2}{2!} + \cdots,
\end{align*}
where $(a)_n = a(a+1)\cdots(a+n-1)$ is the Pochhammer symbol. 

Consequently, for even values of $m$, the solutions presented in \cite[Proposition 5.6]{medvedev2025free} become undefined due to singularities in the parameters of the hypergeometric functions:
\begin{equation*}
z_0(t) = C_1 \frac{\sin^{2-m} t}{2 - m} \, {}_2F_1\left(\frac{3}{2}, 1 - \frac{m}{2}; 2 - \frac{m}{2}; \sin^2 t\right) + C_2,
\end{equation*}
and
\begin{equation*}
z_{1i}(t) = C_1 \frac{\sin^{-m} t}{m} \, {}_2F_1\left(\frac{1}{2}, -\frac{m}{2}; \frac{2-m}{2}; \sin^2 t\right) + C_2.
\end{equation*}

Instead, the general solution to the first equation \eqref{eq:z0} can be robustly expressed in terms of associated Legendre functions, which remain well-defined for all $m$:
\begin{align*}
z_0(t) &= c_1 \left(1 - \cos^2 t\right)^{\frac{2-m}{4}} P_{\frac{m}{2}}^{\frac{m-2}{2}}(\cos t) \sec t \\
&\quad + c_2 \left(1 - \cos^2 t\right)^{\frac{2-m}{4}} Q_{\frac{m}{2}}^{\frac{m-2}{2}}(\cos t) \sec t.
\end{align*}

Similarly, the general solution for the second equation \eqref{eq:a1i} can be expressed as:
\[
a_{1i}(t) = c_1 \left(1 - \cos^2 t\right)^{\frac{2-m}{4}} P_{\frac{m}{2}}^{\frac{m}{2}}(\cos t) + c_2 \left(1 - \cos^2 t\right)^{\frac{2-m}{4}} Q_{\frac{m}{2}}^{\frac{m}{2}}(\cos t),
\]
where $P_l^\mu$ and $Q_l^\mu$ denote the associated Legendre functions of the first and second kind, respectively. Unlike the hypergeometric representation, these functions do not exhibit singularities for integer values of $m$.

An analogous observation holds for the hyperbolic case.

 \bibliography{mybib}
\bibliographystyle{alpha}
\end{document}